\documentclass[11pt,reqno]{amsart}

\usepackage{amssymb,amsmath,amsfonts,amsthm}
\usepackage{mathtools}
\usepackage{enumitem}
\usepackage{graphicx}
\usepackage{microtype}
\usepackage[colorlinks=true,linkcolor=blue,citecolor=blue,urlcolor=blue]{hyperref}

\newtheorem{theorem}{Theorem}[section]
\newtheorem{lemma}[theorem]{Lemma}
\newtheorem{proposition}[theorem]{Proposition}

\theoremstyle{definition}
\newtheorem{definition}[theorem]{Definition}

\theoremstyle{remark}
\newtheorem{remark}[theorem]{Remark}

\numberwithin{equation}{section}

\newcommand{\R}{\mathbb R}
\newcommand{\Sn}{\mathbb S}

\newcommand{\sS}{\mathcal S}
\newcommand{\cQ}{\mathcal Q}

\title[Rotational odd $\sigma_k$-flow tori]
{Singular Rotational Self-Similar Tori for Odd
$\sigma_k$-Curvature Flows}

\author{Haoxuan Cheng}
\thanks{School of Mathematical Sciences, Fudan University, Shanghai 200433,
China. Email address: hxcheng25@m.fudan.edu.cn}
\author{Junqi Lai}
\thanks{School of Mathematical Sciences, South China Normal University,
Guangzhou 510631, China. Email address: 2019021668@m.scnu.edu.cn}
\author{Guoxin Wei}
\thanks{School of Mathematical Sciences, South China Normal University,
Guangzhou 510631, China. Email address: weiguoxin@tsinghua.org.cn}


\begin{document}

\begin{abstract}
For every pair of integers $3\leq k<n$ with $k$ odd, we construct a compact
embedded rotational torus in $\R^{n+1}$ whose homothetic dilations satisfy the
unnormalised $\sigma_k$-curvature flow in a Sobolev almost-everywhere sense.
Its profile curve has H\"older regularity $C^{1,1/k}$ and Sobolev regularity
$W^{2,p}$ for every $1\leq p<k/(k-1)$.  Away from two
singular latitudes the torus is smooth; globally, the flow equation is
interpreted using the weak shape operator of the associated Lipschitz
boundary.  Under rotational symmetry, the self-similar
equation $\langle X,\nu\rangle=-\sigma_k$, where $X$ is the position vector
and $\nu$ is the unit normal, reduces to a degenerate profile system.
We solve this system by combining an odd-power desingularisation, a shooting
argument, uniform radial and axial bounds, and a strict gap between the
shooting parameters and the cylindrical radius.  No classical $C^2$
rotational torus can satisfy the soliton equation, so the loss of regularity
is unavoidable within the rotational toroidal class.
\end{abstract}

\maketitle

\section{Introduction}

Let $M^n$ be a smooth $n$-manifold and let
$X:M^n\to\R^{n+1}$ be an oriented hypersurface immersion into Euclidean
$(n+1)$-space, where $\R$ denotes the real numbers.  Let $\nu$ be its unit
normal.  We use the shape-operator convention
$S=-D\nu$, where $D\nu$ is the Euclidean differential of the normal field.
Fix integers
\begin{equation}\label{eq:nk-range}
 3\leq k<n,\qquad k\ \text{odd},
\end{equation}
and write $\sigma_k(S)$ for the $k$-th elementary symmetric polynomial of
the principal curvatures.  We study the unnormalised $\sigma_k$-curvature flow
\begin{equation}\label{eq:sigma-k-flow}
 \left\langle\frac{\partial X}{\partial t},\nu\right\rangle=\sigma_k(S).
\end{equation}
Fix an extinction time $T\in\R$.  For the homothetically shrinking ansatz
$X_t=((k+1)(T-t))^{1/(k+1)}X$, equation \eqref{eq:sigma-k-flow} reduces to the soliton
equation
\begin{equation}\label{eq:self-similar}
  \langle X,\nu\rangle=-\sigma_k(S).
\end{equation}
Thus our primary geometric problem is to construct a compact rotational
self-similar solution of \eqref{eq:sigma-k-flow}; equation
\eqref{eq:self-similar}, and subsequently its profile ODE, are the stationary
equations used to carry out that construction.

Self-similar solutions are basic models for singularities of geometric flows.
For curve shortening flow, homothetic solutions were studied systematically
by Abresch--Langer~\cite{AbreschLanger1986}.  In mean-curvature flow, the
foundational results of Huisken~\cite{Huisken1984,Huisken1990}, the rotational
shrinking torus of Angenent~\cite{Angenent1992}, the rotational analysis of
Kleene--M{\o}ller~\cite{KleeneMoller2014}, and the stability theory of
Colding--Minicozzi~\cite{ColdingMinicozzi2012} exhibit the distinct roles of
spherical, cylindrical, planar, and nonconvex models.

The classical theory for fully nonlinear curvature equations and flows is
usually developed within convex, elliptic, or admissible regimes; see
Caffarelli--Nirenberg--Spruck~\cite{CaffarelliNirenbergSpruck1985},
Andrews~\cite{Andrews1994}, McCoy~\cite{McCoy2011}, and
Andrews--Langford--McCoy~\cite{AndrewsLangfordMcCoy2013}.  Rigidity for powers
of $\sigma_k$ was obtained by Gao--Li--Ma~\cite{GaoLiMa2018}, while convergence
and self-similar solutions for powers of Gauss curvature were studied by
Andrews--Guan--Ni~\cite{AndrewsGuanNi2016} and
Brendle--Choi--Daskalopoulos~\cite{BrendleChoiDaskalopoulos2017}.  These
admissible regimes do not cover the singular nonconvex torus considered here.

Our shooting construction is motivated by the rotational construction of
Angenent and by related $\lambda$-hypersurface constructions
of Cheng--Wei~\cite{ChengWei2021} and
Cheng--Lai--Wei~\cite{ChengLaiWei2024}.  After rotational symmetry is imposed,
the soliton equation \eqref{eq:self-similar} becomes an ODE for a planar
generating curve, which we call the profile curve.  This reduced system loses
its highest-order term when the profile curve has
a vertical tangent, precisely where a closed toroidal profile curve must pass.

The degeneration has a geometric consequence that distinguishes
\eqref{eq:self-similar} from the mean-curvature case: a classical $C^2$
rotational torus cannot solve the equation.  Indeed, at an extremum of the
axial coordinate all $n-1$ principal curvatures in the rotational directions
vanish, forcing
$\sigma_k=0$ and hence forcing the extremal axial coordinate itself to vanish.
Thus any rotational toroidal solution must leave the classical $C^2$
category; the equation itself forces the loss of regularity.  At each
vertical tangent, the meridional curvature
$\kappa_{\mathrm{mer}}=-d\theta/ds$ and the rotational curvature
$\kappa_{\mathrm{rot}}=\cos\theta/r$ have respective orders
$|s-s_0|^{-(k-1)/k}$ and $|s-s_0|^{1/k}$.  The $k-1$ vanishing rotational
factors therefore compensate for the divergent meridional factor, and
$\sigma_k$ remains bounded.
The next theorem constructs such a torus and identifies the Sobolev class in
which its homothetic dilations satisfy the flow equation.

\begin{theorem}
\label{thm:main}
For every pair of integers satisfying \eqref{eq:nk-range}, there is a compact
embedded rotational torus $\Sigma^n\subset\R^{n+1}$ whose homothetic
dilations form a shrinking solution of the unnormalised $\sigma_k$-curvature
flow \eqref{eq:sigma-k-flow} in the Sobolev almost-everywhere sense defined in
Section~\ref{sec:weak}.  More precisely, $\Sigma$ has the following properties.
\begin{enumerate}[label=\textup{(\roman*)}]
  \item The torus is smooth away from two rotational latitudes and satisfies
  $\langle X,\nu\rangle=-\sigma_k(S)$ there.
  \item If $\Omega$ is the bounded domain enclosed by $\Sigma$ and
  $S_{\mathrm w}$ is the weak shape operator of $\partial\Omega=\Sigma$, then
  $\langle X,\nu\rangle=-\sigma_k(S_{\mathrm w})$ almost everywhere with
  respect to the $n$-dimensional Hausdorff measure $\mathcal H^n$.
  \item Set $X_t=((k+1)(T-t))^{1/(k+1)}X$ and
  $\Sigma_t=X_t(M)$.  If $\nu_t$ and $S_{\mathrm w,t}$ are the outward unit
  normal and weak shape operator of $\Sigma_t$, respectively, then
  $\langle\partial X_t/\partial t,\nu_t\rangle
  =\sigma_k(S_{\mathrm w,t})$ at $\mathcal H^n$-almost every point, for every
  $t<T$.
  \item Its profile curve has H\"older regularity $C^{1,1/k}$ and belongs
  to the Sobolev class $W^{2,p}$ for every $1\leq p<k/(k-1)$.
\end{enumerate}
If $s$ is arclength along the profile curve and $s_0$ is either singular
parameter, then the profile curve is
not $C^2$ at $s_0$, and the magnitude of its meridional principal curvature
is asymptotic to a positive multiple of $|s-s_0|^{-(k-1)/k}$.
\end{theorem}

The proof has four main ingredients.  First, at a vertical
tangent we replace
$\phi=\theta-\pi/2$ by the odd-power variable $u=\phi^k$.  The resulting equation
has nonzero speed whenever $x\neq0$, which gives a unique local continuation
and the sharp $C^{1,1/k}$ regularity.  Second, profile curves determined by
the initial data $(x,r,\theta)=(0,\delta,0)$ with $\delta>0$ small are compared
with an explicit scaled limiting system.  This comparison, together with
support-function estimates after the limiting scale, proves that they cross
$\theta=\pi/2$ and return to the $r$-axis.  Third, a linearised calculation
near the cylindrical radius establishes a strict gap between the shooting
set and the cylindrical radius.  Finally, analytic barriers
give uniform control of the terminal radius and of the axial coordinate.  A
compactness argument at the critical shooting parameter then produces a
profile curve intersecting the $r$-axis at angle $\pi$, and reflection closes
the curve.

The most delicate point is the axial compactness.  A direct comparison from
below with the cylindrical solution would force the critical shooting family too
close to the cylinder.  Instead, the strict cylinder gap restricts the initial
radius to a compact subinterval, while the invariant inequality $h>0$
prevents the slope $q=\tan\theta$ from decreasing before the singular
crossing.  Together with the terminal-radius bound, this yields a global
bound for $x$.

The paper is organised as follows.  Section~\ref{sec:reduction} derives the
rotational system and records the classical obstruction and reflection
symmetry.  Section~\ref{sec:crossing} develops the local desingularisation.
Section~\ref{sec:small} establishes non-emptiness of the shooting set.
Section~\ref{sec:compactness} proves the cylinder gap and the two global
compactness estimates.  Section~\ref{sec:critical} constructs and closes the
critical solution.  Section~\ref{sec:weak} gives the precise Sobolev
interpretation.  Section~\ref{sec:curvature-measure} constructs the associated
rotational $\sigma_k$-density measure and proves a smooth approximation
theorem without a residual defect measure.

\section{Rotational reduction and elementary structure}\label{sec:reduction}

Set
\begin{equation}\label{eq:AB}
 \begin{aligned}
 A&=A_k:=\binom{n-1}{k-1},\\
 B&=B_k:=\binom{n-1}{k},
 \end{aligned}
\end{equation}
and let $r_{n,k}=B^{1/(k+1)}$ denote the radius of the elementary cylindrical
solution.  A rotational hypersurface is parametrised by
\begin{equation}\label{eq:rot-param}
 X(s,\omega)=(x(s),r(s)\omega),\qquad \omega\in\Sn^{n-1},
\end{equation}
where $\Sn^m\subset\R^{m+1}$ denotes the unit $m$-sphere, $s$ is arclength,
$x=x(s)$ is the axial coordinate,
$r=r(s)>0$ is the distance from the rotation axis, and
$\theta=\theta(s)$ is the real tangent angle.  Thus
\begin{equation}\label{eq:tangent-angle}
\frac{dx}{ds}=\cos\theta,\qquad \frac{dr}{ds}=\sin\theta.
\end{equation}

The rotational calculation gives the following reduced system.

\begin{proposition}\label{prop:reduction}
With the unit normal
$\nu=(\sin\theta,-\cos\theta\,\omega)$, equation
\eqref{eq:self-similar} is equivalent, away from $\cos\theta=0$, to
\begin{equation}\label{eq:profile-system}
\begin{cases}
\dfrac{dx}{ds}=\cos\theta,\\
\dfrac{dr}{ds}=\sin\theta,\\[2mm]
A\left(\dfrac{\cos\theta}{r}\right)^{k-1}\dfrac{d\theta}{ds}
=x\sin\theta-r\cos\theta
+B\left(\dfrac{\cos\theta}{r}\right)^k.
\end{cases}
\end{equation}
\end{proposition}

\begin{proof}
The principal curvatures in the convention $S=-D\nu$ are
\begin{equation}\label{eq:principal-curvatures}
 \kappa_1=-\frac{d\theta}{ds},
 \qquad
 \kappa_2=\cdots=\kappa_n=\frac{\cos\theta}{r}.
\end{equation}
Consequently,
\[
 \sigma_k
 =B\left(\frac{\cos\theta}{r}\right)^k
 -A\left(\frac{\cos\theta}{r}\right)^{k-1}\frac{d\theta}{ds},
 \qquad
 \langle X,\nu\rangle=x\sin\theta-r\cos\theta.
\]
Substitution in \eqref{eq:self-similar} gives
\eqref{eq:profile-system}.
\end{proof}

The classical obstruction can be stated precisely as follows.

\begin{proposition}\label{prop:no-C2-torus}
There is no regular $C^2$ simple closed profile curve contained in $r>0$
whose rotational hypersurface satisfies \eqref{eq:self-similar} pointwise.
\end{proposition}

\begin{proof}
The function $x$ attains a maximum and a minimum on a closed profile curve.  At
either extremum, $\cos\theta=0$ and $\sin\theta=\pm1$.  Since the profile curve is
$C^2$, the meridional principal curvature is finite, whereas every principal
curvature in a rotational direction in \eqref{eq:principal-curvatures}
vanishes.  Every term in $\sigma_k$ contains at least $k-1\geq2$ such
principal curvatures, and hence $\sigma_k=0$.  Equation
\eqref{eq:self-similar} gives
$x\sin\theta=0$, so $x=0$ at every axial extremum.  Therefore
$\max x=\min x=0$.  A regular simple closed curve cannot be contained in the
line $x=0$, which gives a contradiction.
\end{proof}

The system has three elementary solutions.  The $r$-axis gives a hyperplane,
the horizontal line $r=r_{n,k}$ gives the cylinder
$\R\times\Sn^{n-1}(r_{n,k})$, where $\Sn^{n-1}(r_{n,k})$ is the round
$(n-1)$-sphere of radius $r_{n,k}$, and a semicircle of radius
\begin{equation}\label{eq:sphere-radius}
 r_{n,k}^*=(A+B)^{1/(k+1)}
\end{equation}
gives a round sphere.  Direct substitution verifies all three assertions.
The cylindrical profile is a noncompact horizontal line, not a simple closed
profile curve, so it is not covered by Proposition~\ref{prop:no-C2-torus}.

\begin{figure}[htbp]
  \centering
  \includegraphics[width=0.92\textwidth]{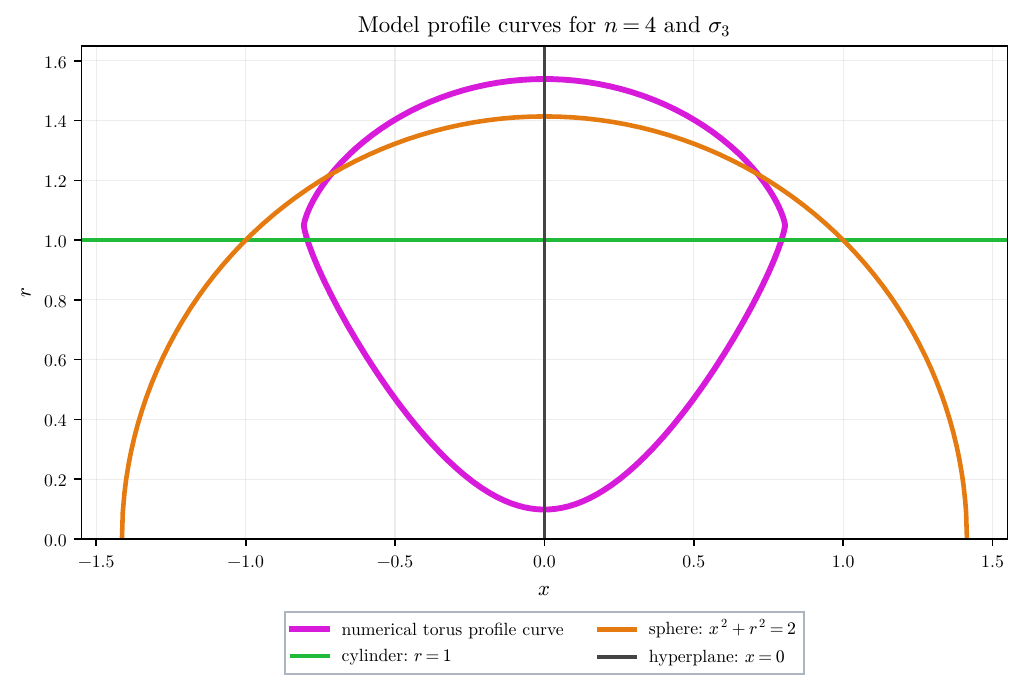}
  \caption{The three exact elementary profile curves and a numerical
  approximation of the closed profile curve for the special case $k=3$ and
  $n=4$.  For this normalisation, the cylinder has $r_{4,3}=1$ and the sphere
  has radius $\sqrt2$.  The magenta curve is included only as a geometric
  illustration; no part of the proof uses its numerical construction.}
  \label{fig:n4-model-profiles}
\end{figure}

The reduced system is invariant under reflection across the $r$-axis.

\begin{proposition}\label{prop:reflection}
If $(x,r,\theta)$ solves \eqref{eq:profile-system} on $[0,L]$, where
$L>0$ is the arclength of the solution segment, then
\begin{equation}\label{eq:reflection}
 (\widetilde x,\widetilde r,\widetilde\theta)(\bar s)
 =(-x(L-\bar s),r(L-\bar s),-\theta(L-\bar s))
\end{equation}
also solves the system.
\end{proposition}

\begin{proof}
Writing $s=L-\bar s$ and differentiating \eqref{eq:reflection} gives
\begin{align*}
 \frac{d\widetilde x}{d\bar s}(\bar s)
 &=\frac{dx}{ds}(s)=\cos\theta(s)=\cos\widetilde\theta(\bar s),\\
 \frac{d\widetilde r}{d\bar s}(\bar s)
 &=-\frac{dr}{ds}(s)=-\sin\theta(s)=\sin\widetilde\theta(\bar s).
\end{align*}
Moreover,
$\frac{d\widetilde\theta}{d\bar s}(\bar s)=\frac{d\theta}{ds}(s)$.
Hence the left-hand side of the
angular equation is unchanged, while its right-hand side becomes
\[
 \widetilde x\sin\widetilde\theta
 -\widetilde r\cos\widetilde\theta
 +B\left(\frac{\cos\widetilde\theta}{\widetilde r}\right)^k
 =
 x\sin\theta-r\cos\theta
 +B\left(\frac{\cos\theta}{r}\right)^k.
\]
Thus all three equations are preserved.
\end{proof}

Thus it suffices to construct a profile curve starting orthogonally from the
$r$-axis and returning to that axis with tangent angle $\pi$.  Its reflection
then gives a closed rotational profile curve.

Whenever $r>0$ and $\cos\theta\neq0$, solving the third equation in
\eqref{eq:profile-system} for $d\theta/ds$ gives a smooth vector field on the
state space with coordinates $(x,r,\theta)$.  Standard ODE theory gives unique
local integral curves of this vector field, their maximal continuation, and
continuous dependence on initial data and parameters; see
Teschl~\cite[Chapter~2]{Teschl2012}.  At $\cos\theta=0$ this vector-field
formulation is singular, so those results do not apply directly.  The next
section constructs the required continuation in an odd-power angular
variable.

We call the image of a solution of \eqref{eq:profile-system} in the state
space $(x,r,\theta)$ its \emph{solution orbit}, or simply its orbit.  Thus
``orbit'' below has its standard dynamical-systems meaning and does not denote
the planar profile curve itself.

\section{Desingularisation at a vertical tangent}\label{sec:crossing}

At $\theta=\pi/2+m\pi$, $m\in\mathbb Z$, the coefficient of $d\theta/ds$ in
\eqref{eq:profile-system} vanishes.  For odd $k$, the correct local variable
is the $k$-th power of the angular displacement.  Throughout this section,
$u^{1/k}$ denotes the real odd root, for both signs of $u$.

For a limiting state $(x_0,r_0,\theta_0)$ at $s=s_0$, we call a local
continuation a \emph{nondegenerate crossing} if
$x,r\in C^1$, $\theta$ is continuous, the profile system holds classically
for $s\neq s_0$, and
\[
 u=(\theta-\theta_0)^k\in C^1,
 \qquad
 \frac{du}{ds}(s_0)\neq0.
\]
The lemma below proves uniqueness in this class and determines the local
regularity.

\begin{lemma}\label{lem:unique-crossing}
Suppose a solution of \eqref{eq:profile-system} approaches, as $s\to s_0$,
\[
 (x,r,\theta)\longrightarrow(x_0,r_0,\theta_0),
 \qquad r_0>0,
 \qquad \theta_0\in\{\pi/2,-\pi/2\},
\]
and assume
\begin{equation}\label{eq:cross-nondegenerate}
 x_0\sin\theta_0\neq0.
\end{equation}
Then the solution has a unique local continuation through $s_0$ in the class of
nondegenerate crossings.  If $u=(\theta-\theta_0)^k$, then $u$ is strictly
monotone near $s_0$, and
\begin{equation}\label{eq:cross-asymptotic}
 \theta(s)-\theta_0
 \sim
 \left(\frac{k r_0^{k-1}x_0\sin\theta_0}{A}\right)^{1/k}
 (s-s_0)^{1/k},
\end{equation}
where the real $k$-th root is understood.  In particular,
\begin{align}
 x(s)-x_0&=O(|s-s_0|^{(k+1)/k}),\label{eq:x-cross-reg}\\
 r(s)-r_0&=O(|s-s_0|),\label{eq:r-cross-reg}\\
 \left|\frac{d\theta}{ds}(s)\right|&\sim
 \frac1k\left|\frac{k r_0^{k-1}x_0\sin\theta_0}{A}\right|^{1/k}
 |s-s_0|^{-(k-1)/k}.\label{eq:curvature-blowup}
\end{align}
Writing $\gamma=(x,r)$, there are constants $C>0$ and $\eta>0$ such that
\begin{equation}\label{eq:tangent-holder}
 \left|
 \frac{d\gamma}{ds}(s)-\frac{d\gamma}{ds}(t)
 \right|
 \leq C|s-t|^{1/k}
 \qquad
 \text{whenever }|s-s_0|<\eta\text{ and }|t-s_0|<\eta.
\end{equation}
Consequently the profile curve is $C^{1,1/k}$ near the crossing, and the
exponent $1/k$ is optimal.  It is also locally $W^{2,p}$ for every
$1\leq p<k/(k-1)$, but it is not $W^{2,k/(k-1)}$ at the crossing.
\end{lemma}

\begin{proof}
\medskip\noindent\textbf{Step 1: construct the regular $u$-equation.}
Set
\[
 \varepsilon=\sin\theta_0\in\{-1,1\},
 \qquad \phi=\theta-\theta_0.
\]
Then
\[
 \cos(\theta_0+\phi)=-\varepsilon\sin\phi,
 \qquad
 \sin(\theta_0+\phi)=\varepsilon\cos\phi,
\]
and \eqref{eq:profile-system} becomes
\begin{equation}\label{eq:phi-system}
 \begin{aligned}
 \frac{dx}{ds}&=-\varepsilon\sin\phi,\qquad
 \frac{dr}{ds}=\varepsilon\cos\phi,\\
 \frac{d\phi}{ds}&=\frac{\varepsilon r^{k-1}}{A\sin^{k-1}\phi}
 \left(x\cos\phi+r\sin\phi-B\frac{\sin^k\phi}{r^k}\right).
 \end{aligned}
\end{equation}
For $u=\phi^k$ define
\begin{equation}\label{eq:H-def}
 H(x,r,u)=
 \begin{cases}
 \displaystyle
 \frac{k\varepsilon r^{k-1}}{A}
 \left(\frac{u^{1/k}}{\sin(u^{1/k})}\right)^{k-1}
 \left[x\cos(u^{1/k})+r\sin(u^{1/k})\right.\\[-1mm]
 \displaystyle\hfill\left.
 -B\frac{\sin^k(u^{1/k})}{r^k}\right],&u\neq0,\\[3mm]
 \displaystyle \frac{k\varepsilon r^{k-1}x}{A},&u=0.
 \end{cases}
\end{equation}
For $u\ne0$, equation \eqref{eq:phi-system} gives
$du/ds=H(x,r,u)$.
The quotients
\[
 \left(\frac{\phi}{\sin\phi}\right)^{k-1},\qquad
 \frac{\sin\phi}{\phi},\qquad
 \frac{\sin^k\phi}{\phi^k}
\]
have removable limits at $\phi=0$.  Therefore $H$ is continuous in
$(x,r,u)$ and locally Lipschitz in $(x,r)$, uniformly for $u$ in a compact
interval, as long as $r$ stays bounded away from zero.  Moreover,
\[
 H(x_0,r_0,0)
 =\frac{k r_0^{k-1}x_0\sin\theta_0}{A}\neq0.
\]
Hence, on a sufficiently small neighbourhood with $r>0$, $H$ has fixed sign
and satisfies $|H|\geq c_0>0$.  Using $u$ as the independent variable, the
system becomes
\begin{equation}\label{eq:u-chart}
 \frac{dx}{du}=-\frac{\varepsilon\sin(u^{1/k})}{H(x,r,u)},
 \qquad
 \frac{dr}{du}=\frac{\varepsilon\cos(u^{1/k})}{H(x,r,u)},
 \qquad
 \frac{ds}{du}=\frac1{H(x,r,u)}.
\end{equation}
\medskip\noindent\textbf{Step 2: existence and uniqueness through the crossing.}
Its right-hand side is continuous in $u$ and locally Lipschitz in the state
variables $(x,r,s)$, uniformly on a smaller neighbourhood.  Standard
nonautonomous ODE theory gives a unique solution through
$(x_0,r_0,s_0)$ at $u=0$.  Since $ds/du=1/H$ is continuous, nonzero, and of
fixed sign, $s(u)$ is strictly monotone and has a local inverse.  Composing
with this inverse produces a two-sided continuation in the original
arclength parameter.  Conversely, every continuation in the stated
nondegenerate class satisfies $du/ds=H$ away from $s_0$.  Since both sides
extend continuously to $s_0$, the equality also holds there.  It is bounded
away from zero after the interval is reduced.  Hence $u$ is monotone and the
continuation solves the same initial-value problem \eqref{eq:u-chart} after
reparametrisation.  This proves uniqueness in the claimed class.

\medskip\noindent\textbf{Step 3: crossing asymptotics and regularity.}
Continuity of $H$ gives
\[
 u(s)=
 \frac{k r_0^{k-1}x_0\sin\theta_0}{A}(s-s_0)
 {}+o(|s-s_0|),
\]
which proves \eqref{eq:cross-asymptotic}.  Equations
\eqref{eq:x-cross-reg} and \eqref{eq:r-cross-reg} follow by integration of the
first two equations in \eqref{eq:phi-system}.

We next prove the two-point H\"older estimate.  Along the continued solution,
$du/ds=H(x(s),r(s),u(s))$ is continuous.  After reducing the parameter
interval if necessary, it therefore satisfies $|du/ds|\leq M$ for some
constant $M>0$.  Hence
\begin{equation}\label{eq:u-lipschitz}
 |u(s)-u(t)|\leq M|s-t|.
\end{equation}
For arbitrary real numbers $a$ and $b$, let $\alpha=a^{1/k}$ and
$\beta=b^{1/k}$ denote their
real $k$-th roots.  If $\alpha\beta\geq0$, interchange $\alpha$ and $\beta$
if necessary so that $|\alpha|\geq|\beta|$.  Factorisation gives
\[
 \begin{aligned}
 |\alpha^k-\beta^k|
 &=(|\alpha|-|\beta|)
 \sum_{j=0}^{k-1}|\alpha|^{k-1-j}|\beta|^j\\
 &\geq(|\alpha|-|\beta|)^k
 =|\alpha-\beta|^k.
 \end{aligned}
\]
If $\alpha\beta<0$, convexity gives
\[
 |\alpha-\beta|^k
 =(|\alpha|+|\beta|)^k
 \leq2^{k-1}(|\alpha|^k+|\beta|^k)
 =2^{k-1}|\alpha^k-\beta^k|.
\]
Thus, in both cases,
\begin{equation}\label{eq:odd-root-holder}
 |a^{1/k}-b^{1/k}|
 \leq2^{(k-1)/k}|a-b|^{1/k}.
\end{equation}
Since $\phi=u^{1/k}$, equations \eqref{eq:u-lipschitz} and
\eqref{eq:odd-root-holder} give
\[
 |\theta(s)-\theta(t)|
 \leq(2^{k-1}M)^{1/k}|s-t|^{1/k}.
\]
Since the curve is parametrised by arclength,
$d\gamma/ds=(\cos\theta,\sin\theta)$, and the chord-length identity gives
\[
 \left|
 \frac{d\gamma}{ds}(s)-\frac{d\gamma}{ds}(t)
 \right|
 =2\left|\sin\frac{\theta(s)-\theta(t)}2\right|
 \leq|\theta(s)-\theta(t)|,
\]
which proves \eqref{eq:tangent-holder}.  Taking $t=s_0$, the nonzero
coefficient in \eqref{eq:cross-asymptotic}, the chord-length identity, and
$\sin z/z\to1$ show that the left-hand side is asymptotic to a positive
multiple of $|s-s_0|^{1/k}$.  Thus no exponent greater than $1/k$ is possible.

Finally, $du/ds=k\phi^{k-1}\,d\phi/ds$ gives
\eqref{eq:curvature-blowup}.  Since the norm of the
second derivative of an arclength-parametrised plane curve is $|d\theta/ds|$,
local $L^p$ integrability is equivalent to $(k-1)p/k<1$.  At
$p=k/(k-1)$ one obtains
the divergent integral of $|s-s_0|^{-1}$.
\end{proof}

\begin{figure}[htbp]
  \centering
  \includegraphics[width=\textwidth]{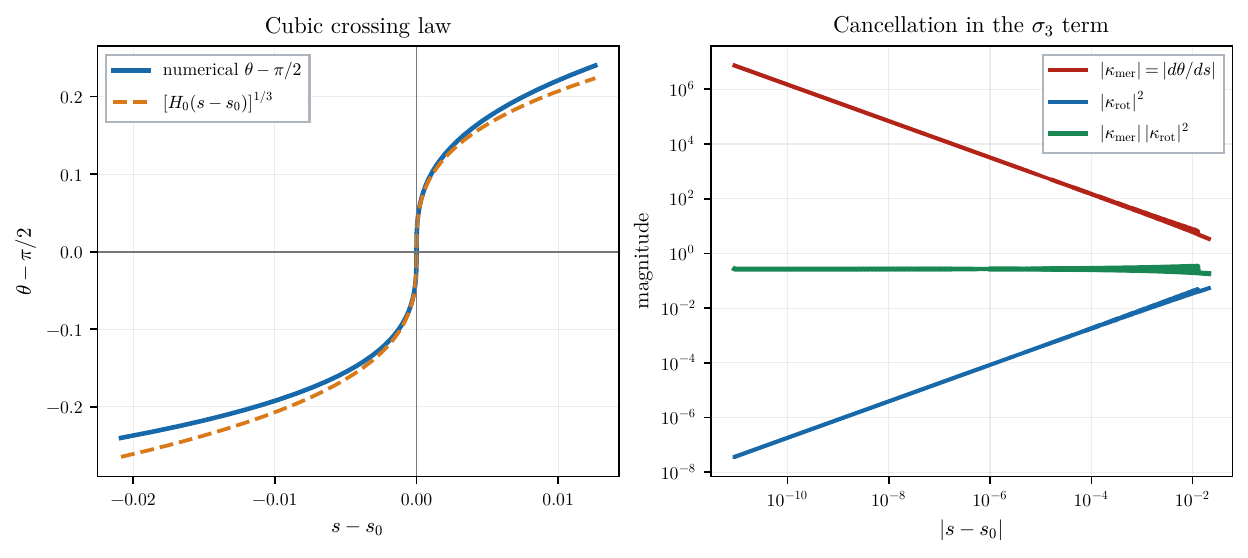}
  \caption{Numerical illustration of the local crossing laws in the special
  case $k=3$ and $n=4$.  The left panel shows
  $\theta-\pi/2\sim[H_0(s-s_0)]^{1/3}$, where
  $H_0:=H(x_0,r_0,0)$.  The right panel shows the blow-up of the meridional
  principal curvature $\kappa_{\mathrm{mer}}=-d\theta/ds$, the vanishing of
  $\kappa_{\mathrm{rot}}^2$, where
  $\kappa_{\mathrm{rot}}=\cos\theta/r$ is the principal curvature in each
  rotational direction, and the bounded product that enters $\sigma_3$.
  This figure visualises the asymptotics proved above and is not used to
  establish them.}
  \label{fig:n4-crossing-scaling}
\end{figure}

The local continuation also depends continuously on the incoming data.

\begin{lemma}\label{lem:crossing-map}
Let
\begin{equation}\label{eq:crossing-compact-data}
 K_0\subset\{(x,r):x>0,\ r>0\}
\end{equation}
be a compact set of crossing data at $\theta=\pi/2$.  There is $u_0>0$ such
that every solution whose data at $u=0$ belong to $K_0$ is represented in the
$u$-chart on the full interval $[-u_0,u_0]$.  On a neighbourhood of the
corresponding incoming data, the map from the section $u=-u_0$ to the section
$u=u_0$, including the elapsed arclength, depends continuously on the
incoming data.
\end{lemma}

\begin{proof}
At $u=0$, equation \eqref{eq:H-def} gives
\[
 H(x,r,0)=\frac{k r^{k-1}x}{A}
 \geq \min_{K_0}\frac{k r^{k-1}x}{A}>0
\]
on $K_0$.  Choose an open neighbourhood $K$ of $K_0$ whose closure is a
compact subset of $\{x>0,r>0\}$.  By continuity and compactness, after
shrinking $K$ if necessary, there are $u_0>0$ and $h_0>0$ such that
\[
 H(x,r,u)\geq h_0
 \qquad ((x,r)\in\overline K,\ |u|\leq u_0).
\]
The right-hand side of \eqref{eq:u-chart} is uniformly bounded on
$\overline K\times[-u_0,u_0]$.  Since $K_0$ has positive distance from
$\partial K$, a further uniform reduction of $u_0$ ensures that every
solution with initial state in $K_0$ at $u=0$ remains in $K$ throughout
$[-u_0,u_0]$.

On the fixed product set $\overline K\times[-u_0,u_0]$, the right-hand side
of \eqref{eq:u-chart} is
continuous in $u$ and uniformly locally Lipschitz in $(x,r,s)$.
The flow from $u=0$ to $u=-u_0$ maps $K_0$ to a compact set of incoming
data.  Standard continuous dependence gives a neighbourhood of that compact
set on which solutions exist throughout $[-u_0,u_0]$ and depend continuously
on their incoming data.  Evaluation at $u=u_0$ gives continuity of the
outgoing $(x,r)$ data.  The elapsed arclength is
\[
 s(u_0)-s(-u_0)
 =\int_{-u_0}^{u_0}\frac{du}{H(x(u),r(u),u)},
\]
and is continuous by the same solution dependence and the common lower
bound $H\geq h_0$.
\end{proof}

\begin{remark}
The map $\phi\mapsto u=\phi^k$ is a desingularisation, not a smooth change of
coordinates: its inverse is not $C^1$ at the origin.  The nonzero value of
$du/ds(s_0)$ is what makes the continuation unique while allowing the curvature
to blow up.
\end{remark}

\section{Small-radius shooting orbits}\label{sec:small}

For $0<\delta<r_{n,k}$, let $(x_\delta,r_\delta,\theta_\delta)$ denote the
solution with initial data
\begin{equation}\label{eq:small-IVP}
 x_\delta(0)=0,\qquad r_\delta(0)=\delta,\qquad
 \theta_\delta(0)=0.
\end{equation}
Since
\begin{equation}\label{eq:initial-curvature}
 \frac{d\theta_\delta}{ds}(0)=\frac{B-\delta^{k+1}}{A\delta}
 =\frac{n-k}{k\delta}+O(\delta^k),
\end{equation}
the natural initial scale is $s=\delta\tau$.

\subsection{The scaled limit}

Put
\begin{equation}\label{eq:scaled-variables}
 x_\delta(\delta\tau)=\delta\xi(\tau),\qquad
 r_\delta(\delta\tau)=\delta\rho(\tau),\qquad
 \theta_\delta(\delta\tau)=\vartheta(\tau),
\end{equation}
and set
\begin{equation}\label{eq:lambda}
 \lambda=\frac BA=\frac{n-k}{k}>0.
\end{equation}
Here $\tau$ is scaled arclength, while $\xi$, $\rho$, and $\vartheta$ are,
respectively, the scaled axial coordinate, scaled radius, and tangent angle.
The scaled equations are
\begin{equation}\label{eq:scaled-system}
 \frac{d\xi}{d\tau}=\cos\vartheta,\qquad
 \frac{d\rho}{d\tau}=\sin\vartheta,\qquad
 \frac{d\vartheta}{d\tau}=\lambda\frac{\cos\vartheta}{\rho}
 +\frac{\delta^{k+1}\rho^{k-1}}{A\cos^{k-1}\vartheta}
 (\xi\sin\vartheta-\rho\cos\vartheta).
\end{equation}

The first scaled estimate compares the small-radius solutions with the
limiting system on a fixed radial section.

\begin{lemma}\label{lem:scaled-asymptotics}
Fix $\rho_*>1$ and define
\begin{equation}\label{eq:limit-data}
 \vartheta_*=\arccos(\rho_*^{-\lambda}),
 \quad
 \tau_*=\int_1^{\rho_*}\frac{d\rho}
 {\sqrt{1-\rho^{-2\lambda}}},
 \quad
 \xi_*=\int_1^{\rho_*}\frac{\rho^{-\lambda}\,d\rho}
 {\sqrt{1-\rho^{-2\lambda}}}.
\end{equation}
Then, as $\delta\downarrow0$,
\begin{align}
 x_\delta(\delta\tau_*)&=\delta\xi_*+O(\delta^{k+2}),\label{eq:x-scaled-asymp}\\
 r_\delta(\delta\tau_*)&=\delta\rho_*+O(\delta^{k+2}),\label{eq:r-scaled-asymp}\\
 \theta_\delta(\delta\tau_*)&=\vartheta_*+O(\delta^{k+1}).
 \label{eq:theta-scaled-asymp}
\end{align}
\end{lemma}

\begin{proof}
On compact subsets on which $\cos\vartheta$ is bounded away from zero,
\eqref{eq:scaled-system} converges in $C^1$ to
\begin{equation}\label{eq:limit-system}
 \frac{d\xi}{d\tau}=\cos\vartheta,\qquad
 \frac{d\rho}{d\tau}=\sin\vartheta,\qquad
 \frac{d\vartheta}{d\tau}=\lambda\frac{\cos\vartheta}{\rho}.
\end{equation}
Along its solution through $(0,1,0)$,
\[
 \frac{d}{d\rho}\log(\cos\vartheta)=-\frac\lambda\rho,
 \qquad\text{hence}\qquad
 \cos\vartheta=\rho^{-\lambda}.
\]
Integration gives exactly \eqref{eq:limit-data}.  The limiting orbit on
$[0,\tau_*]$ satisfies
\[
 1\leq\rho\leq\rho_*,
 \qquad
 \cos\vartheta\geq \rho_*^{-\lambda}=:c_*>0.
\]
Choose a compact tubular neighbourhood $K$ of this orbit on which
$\rho\geq1/2$ and $\cos\vartheta\geq c_*/2$.  If $F_\delta$ and $F_0$
denote the vector fields in \eqref{eq:scaled-system} and
\eqref{eq:limit-system}, respectively, then on $K$
\[
 \|F_\delta-F_0\|\leq C\delta^{k+1},
 \qquad
 \|D F_0\|\leq L.
\]
Here $DF_0$ is the Jacobian of $F_0$ with respect to the state variables, and
the displayed norms are the corresponding Euclidean and operator norms.
Let $\tau_{\mathrm e}$ be the first exit parameter of the $\delta$-orbit from
$K$, truncated at $\tau_*$.  Since the two solutions have the same initial
data, the integral equations give, for $0\leq\tau\leq\tau_{\mathrm e}$,
\[
 |Y_\delta(\tau)-Y_0(\tau)|
 \leq C\delta^{k+1}\tau+
 L\int_0^\tau|Y_\delta(\sigma)-Y_0(\sigma)|\,d\sigma,
\]
where $Y_\delta=(\xi_\delta,\rho_\delta,\vartheta_\delta)$ is the perturbed
scaled solution and $Y_0=(\xi_0,\rho_0,\vartheta_0)$ is the limiting scaled
solution.  The integral form of Gronwall's inequality yields
\[
 \sup_{0\leq\tau\leq\tau_{\mathrm e}}
 |Y_\delta(\tau)-Y_0(\tau)|\leq C_*\delta^{k+1}.
\]
For sufficiently small $\delta$ this is smaller than the distance from the
limiting orbit to $\partial K$, so $\tau_{\mathrm e}=\tau_*$.  Evaluating at
$\tau_*$ gives an $O(\delta^{k+1})$ error in all three scaled variables.
Multiplication by $\delta$ in the first two components of
\eqref{eq:scaled-variables} proves
\eqref{eq:x-scaled-asymp}--\eqref{eq:theta-scaled-asymp}.
\end{proof}

\subsection{Passage through the singular angle}

Define
\begin{equation}\label{eq:Q-def}
 Q=x\sin\theta-r\cos\theta.
\end{equation}
Here $X$ is the position vector of the rotational immersion, so
$Q=\langle X,\nu\rangle$ is its support function.  Isolating it separates
the geometric support term from the
purely rotational degree-$k$ term in the angular equation; its derivative also
has the simple sign identity below.
Whenever $0<\theta<\pi/2$ and $d\theta/ds>0$, direct differentiation gives
\begin{equation}\label{eq:Q-prime}
 \frac{dQ}{ds}=\frac{d\theta}{ds}
 (x\cos\theta+r\sin\theta)>0.
\end{equation}

The support-function identity forces every sufficiently small orbit to reach
the singular angle.

\begin{lemma}\label{lem:small-crosses}
For all sufficiently small $\delta>0$, the orbit
\eqref{eq:small-IVP} reaches $\theta=\pi/2$ at a finite arclength parameter
$s_\delta^+$.  Before that parameter value,
\[
 0<\theta_\delta<\pi/2,\qquad x_\delta>0,\qquad r_\delta>0.
\]
Moreover, for a constant $C=C(n,k)$,
\begin{equation}\label{eq:crossing-bounds}
 r_\delta(s_\delta^+)\leq C,
 \qquad
 x_\delta(s_\delta^+)\longrightarrow0.
\end{equation}
The crossing is nondegenerate.
\end{lemma}

\begin{proof}
\medskip\noindent\textbf{Step 1: qualitative crossing.}
Choose a fixed $\rho_0>1$ so large that the limiting quantity
\begin{equation}\label{eq:limit-Q}
 q_{\mathrm{lim}}:=\xi(\rho_0)\sqrt{1-\rho_0^{-2\lambda}}
 -\rho_0^{1-\lambda}
\end{equation}
is positive.  The subscript records that this is the support function of the
limiting scaled orbit at the fixed section $\rho=\rho_0$; it is unrelated to
the later slope lower bound.  This choice is possible because its leading behaviour is a positive
multiple of $\rho^{1-\lambda}$ when $0<\lambda<1$, is $\log\rho-1$ when
$\lambda=1$, and tends to a positive constant when $\lambda>1$.
At the terminal scaled parameter in Lemma~\ref{lem:scaled-asymptotics}, the perturbed
radial coordinate is $\rho_0+O(\delta^{k+1})$.  Since the limiting terminal angle
is strictly positive, $d\rho/d\tau=\sin\vartheta$ has a common positive lower
bound in a neighbourhood of that point.  Consequently the perturbed orbit
meets the exact section $\rho=\rho_0$ at a unique parameter value
$\tau_{0,\delta}=\tau_0+O(\delta^{k+1})$, and continuous dependence on this
$O(\delta^{k+1})$ parameter interval preserves all three
$O(\delta^{k+1})$ scaled errors.
Put $s_0=\delta\tau_{0,\delta}$.  Then
\begin{equation}\label{eq:Q-start}
 Q(s_0)=\delta q_{\mathrm{lim}}+O(\delta^{k+2})
 \geq\frac{q_{\mathrm{lim}}}{2}\delta>0.
\end{equation}
Equations \eqref{eq:Q-prime} and \eqref{eq:profile-system} imply that $Q$
and $\theta$ remain strictly increasing as long as $\theta<\pi/2$: indeed,
$Q>0$ makes the right-hand side of the angular equation positive.  Since
$dr/ds=\sin\theta>0$ there, $r$ is a valid parameter and
\begin{equation}\label{eq:cos-power}
 \frac{d}{dr}\cos^k\theta
 \leq-\frac{kQ(s_0)}A r^{k-1}.
\end{equation}
If the orbit stayed in $0<\theta<\pi/2$, integration would give
\[
 \cos^k\theta(r)
 \leq \cos^k\theta(s_0)
 -\frac{Q(s_0)}A\bigl(r^k-r^k(s_0)\bigr),
\]
whose right-hand side becomes negative at a finite radius.  The solution
cannot terminate earlier while $r$ remains below that radius and
$\theta$ remains in a compact subinterval of $(0,\pi/2)$, by the regular ODE
continuation theorem.  Hence it reaches $\theta=\pi/2$ at a finite arclength
parameter.

\medskip\noindent\textbf{Step 2: comparison up to a growing target scale.}
It remains to obtain bounds uniform as $\delta\to0$.  Let
\begin{equation}\label{eq:rho-delta}
 \rho_\delta=\delta^{-a},\qquad
 a=\bigl(1+(k-1)\lambda\bigr)^{-1}.
\end{equation}
Before the crossing use $\rho$ as independent variable and set
\begin{equation}\label{eq:zQscaled}
 z=\cos\theta_\delta,\qquad
 \cQ=\xi\sqrt{1-z^2}-\rho z,\qquad W=\rho^\lambda z.
\end{equation}
Here $z$ is the horizontal tangent component, so the vertical crossing is
exactly $z=0$; $\cQ=Q/\delta$ is the scale-free support function; and $W$
measures the deviation from the limiting first integral
$z=\rho^{-\lambda}$, for which $W\equiv1$.  We later differentiate $W^k$
because the $k$-th power cancels the apparent $z^{-(k-1)}$ singularity in the equation for
$W$.
The exact scaled equations are
\begin{align}\label{eq:growing-system}
 \frac{d\xi}{d\rho}&=\frac{z}{\sqrt{1-z^2}},
 \qquad
 \frac{dz}{d\rho}=-\lambda\frac z\rho
 -\frac{\delta^{k+1}\rho^{k-1}}{Az^{k-1}}\cQ,
 \notag\\
 \frac{dW}{d\rho}&=-\frac{\delta^{k+1}}{A}
 \rho^{\lambda+k-1}z^{-(k-1)}\cQ.
\end{align}
By \eqref{eq:Q-prime}, $\cQ\geq q_{\mathrm{lim}}/2$.  In particular, the last equation
in \eqref{eq:growing-system} shows that $W$ is decreasing.  Moreover,
Lemma~\ref{lem:scaled-asymptotics} gives
$W(\rho_0)=1+O(\delta^{k+1})$.  Enlarge the already fixed $\rho_0$, if necessary,
so that $2\rho_0^{-\lambda}<1$.  For sufficiently small $\delta$, as long as
the orbit has not crossed $\theta=\pi/2$,
\[
 0<W(\rho)\leq W(\rho_0)\leq2,
 \qquad
 z(\rho)\leq2\rho^{-\lambda}.
\]
Thus $\sin\theta=\sqrt{1-z^2}$ has a fixed positive lower bound.  Since
$0<\cQ=\xi\sqrt{1-z^2}-\rho z\leq\xi$, it follows that
\[
 \xi(\rho)\leq C
 \begin{cases}
 \rho^{1-\lambda},&0<\lambda<1,\\
 \log\rho,&\lambda=1,\\
 1,&\lambda>1,
 \end{cases}
\]
The apparent singular factor $z^{-(k-1)}$ disappears after taking the
$k$-th power of $W$:
\begin{equation}\label{eq:W-power}
 \frac{d}{d\rho}W^k
 =-\frac{k\delta^{k+1}}{A}\rho^{k\lambda+k-1}\cQ.
\end{equation}
Integrating this identity up to any point before the crossing or the target
value $\rho_\delta$ gives
\begin{equation}\label{eq:W-error}
 0\leq W^k(\rho_0)-W^k(\rho)\leq C\delta^{k+1}
 \begin{cases}
 \rho^{k+1+(k-1)\lambda},&0<\lambda<1,\\
 \rho^{2k}\log\rho,&\lambda=1,\\
 \rho^{k(1+\lambda)},&\lambda>1.
 \end{cases}
\end{equation}
At the target scale $\rho=\rho_\delta$ the three upper bounds are respectively
\[
 O\left(\delta^{\frac{k(k-1)\lambda}{1+(k-1)\lambda}}\right),\qquad
 O(\delta^{k-1}|\log\delta|),\qquad
 O\left(\delta^{\frac{1+(k^2-k-1)\lambda}{1+(k-1)\lambda}}\right),
\]
and hence tend to zero.  If the orbit crossed $\theta=\pi/2$ before reaching
$\rho_\delta$, then $W\to0$ at that crossing.  Letting $\rho$ tend to the
crossing value in \eqref{eq:W-error} would give
$W^k(\rho_0)=o(1)$, contradicting $W(\rho_0)=1+O(\delta^{k+1})$.
Consequently the orbit reaches $\rho_\delta$ first.  The same estimate is
uniform on the entire preceding interval and gives
\begin{equation}\label{eq:W-one}
 z=\rho^{-\lambda}(1+\eta_\delta(\rho)),
 \qquad
 \sup_{\rho_0\leq\rho\leq\rho_\delta}
 |\eta_\delta(\rho)|\longrightarrow0
\end{equation}
uniformly up to that point.  It follows that
\begin{equation}\label{eq:cQ-lower}
 \cQ(\rho_\delta)\geq
 \begin{cases}
 c\rho_\delta^{1-\lambda},&0<\lambda<1,\\
 c\log\rho_\delta,&\lambda=1,\\
 c,&\lambda>1.
 \end{cases}
\end{equation}
for a constant $c>0$ independent of $\delta$.  We include the calculation.
Having first fixed $\rho_0$ as above, enlarge it if necessary so that
$2\rho_0^{-\lambda}<1$.  On $[\rho_0,\rho_\delta]$,
\eqref{eq:growing-system} and \eqref{eq:W-one} give
\begin{equation}\label{eq:xi-growing-asymptotic}
 \frac{d\xi}{d\rho}
 =\frac{\rho^{-\lambda}(1+\eta_\delta)}
 {\sqrt{1-\rho^{-2\lambda}(1+\eta_\delta)^2}}
 =\rho^{-\lambda}
 \left(1+O(\|\eta_\delta\|_\infty)+O(\rho^{-2\lambda})\right).
\end{equation}
\medskip\noindent\textbf{Case 1: $0<\lambda<1$.}
Integration yields
\[
 \xi(\rho_\delta)
 =\left(\frac{1}{1-\lambda}
 +o_\delta(1)+o_{\rho_\delta}(1)\right)
 \rho_\delta^{1-\lambda}.
\]
Here the second error includes the fixed lower endpoint and the integral of
$O(\rho^{-3\lambda})$, both of which are
$o(\rho_\delta^{1-\lambda})$.  Since
\[
 \sqrt{1-z^2}=1+o(1),
 \qquad
 \rho_\delta z(\rho_\delta)
 =(1+o(1))\rho_\delta^{1-\lambda},
\]
we obtain
\[
 \cQ(\rho_\delta)
 =\left(\frac{\lambda}{1-\lambda}+o(1)\right)
 \rho_\delta^{1-\lambda},
\]
which gives the first line of \eqref{eq:cQ-lower}.

\medskip\noindent\textbf{Case 2: $\lambda=1$.}
Equation \eqref{eq:xi-growing-asymptotic} instead gives
\[
 \xi(\rho_\delta)=(1+o(1))\log\rho_\delta,
 \qquad
 \rho_\delta z(\rho_\delta)=1+o(1),
\]
and hence the second line.

\medskip\noindent\textbf{Case 3: $\lambda>1$.}
The monotonicity $Q\geq Q(s_0)$ already gives
$\cQ=Q/\delta\geq q_{\mathrm{lim}}/2$, proving the third line.  Thus the constant $c$ is
obtained by first fixing $\rho_0$ and then taking $\delta$ sufficiently small.

\medskip\noindent\textbf{Step 3: force the crossing after the target.}
At $r=\delta\rho_\delta$, put
$r_t=\delta\rho_\delta$, $z_t=z(\rho_\delta)>0$, and
$Q_\delta=\delta\cQ(\rho_\delta)>0$, and write $x_t>0$ for the
$x$-coordinate of this target point.  Thus $Q_\delta$ is not a new function:
it is the support function $Q$ evaluated at the target.  Its purpose is to
pair the quantitative lower bound for $Q$ with the small value of $z_t$ and
thereby obtain a crossing-radius bound uniform in $\delta$.  We first verify that the subsequent
crossing used below actually occurs.  At the target $Q=Q_\delta$.  As long
as $Q>0$ on the regular pre-crossing branch, the angular equation gives
$d\theta/ds>0$; then \eqref{eq:Q-prime} gives $dQ/ds>0$.  Thus $Q$ cannot have a
first return to zero and, in fact, $Q\geq Q_\delta$.  Consequently
\[
 \sin\theta\geq\sqrt{1-z_t^2}>0
\]
throughout this segment.  Direct substitution of the angular equation gives,
at every regular pre-crossing point,
\begin{equation}\label{eq:post-target-power-derivative}
\begin{aligned}
 \frac{d}{dr}z^k
 &=-kz^{k-1}\frac{d\theta}{ds}\\
 &=-\frac{k}{A}r^{k-1}Q-\frac{kB}{A}\frac{z^k}{r}\\
 &\leq-\frac{kQ_\delta}{A}r^{k-1}.
\end{aligned}
\end{equation}
Here the last line uses $Q\geq Q_\delta$ together with
$B>0$, $r>0$, and $z>0$.  Integrating
\eqref{eq:post-target-power-derivative} only to an arbitrary regular
pre-crossing point gives
\begin{equation}\label{eq:post-target-power}
 z^k(r)\leq z_t^k-\frac{Q_\delta}{A}(r^k-r_t^k).
\end{equation}
Set
\[
 r_\sharp^k=r_t^k+\frac{A z_t^k}{Q_\delta}.
\]
This is precisely the radius at which the right-hand side of
\eqref{eq:post-target-power} vanishes; it is introduced only as a finite
barrier forcing $z$ to reach zero.
The branch cannot terminate at a finite parameter while $r<r_\sharp$ and
$z>0$.
Indeed, an infinite endpoint is excluded by the displayed positive lower
bound for $dr/ds=\sin\theta$.  At a finite endpoint, $x$ and $r$ are bounded
with $x\geq x_t>0$ and $r\geq r_t>0$; if $z$ stays away from zero,
ordinary ODE continuation applies, while $z\to0$ is precisely the prescribed
nondegenerate crossing.  Thus, if no crossing occurred first, the branch
would reach $r=r_\sharp$.  If $z\to0$ there, this is already the desired
crossing.  Otherwise, taking the regular-side limit in
\eqref{eq:post-target-power} gives $z^k(r_\sharp)\leq0$, contrary to
$z(r_\sharp)>0$.  The crossing therefore exists and occurs no later than
$r_\sharp$.  Since
$z_t^k=\rho_\delta^{-k\lambda}(1+o(1))$, this proves
\begin{equation}\label{eq:r-cross-bound}
 r_\delta^k(s_\delta^+)
 \leq(\delta\rho_\delta)^k
 +A\frac{\rho_\delta^{-k\lambda}}{Q_\delta}(1+o(1)).
\end{equation}
For $0<\lambda<1$, the last quotient is bounded because
$\delta^{-1}\rho_\delta^{-(1+(k-1)\lambda)}=1$; for $\lambda=1$ it is even
smaller, and for $\lambda>1$ one has
\[
 \delta^{-1}\rho_\delta^{-k\lambda}
 =\delta^{(\lambda-1)/(1+(k-1)\lambda)}\longrightarrow0.
\]
This proves the radius bound.

\medskip\noindent\textbf{Step 4: bounds for the crossing coordinates.}
Before $r=\delta\rho_\delta$, the preceding estimate for $\xi$ gives
$x_\delta=\delta\xi=o(1)$ in each of the three regimes.  Afterwards,
$dx/dr=\cot\theta$.  Since $\theta$ is increasing, $\cot\theta$ is
nonnegative and no larger than its value at
$r=\delta\rho_\delta$; by \eqref{eq:W-one}, that value tends to zero.
The remaining $r$-length is uniformly bounded by
\eqref{eq:r-cross-bound}.  Hence the additional change in $x$ is $o(1)$,
proving
\eqref{eq:crossing-bounds}.  Since $dx/ds>0$ before the crossing,
$x_\delta(s_\delta^+)>0$, and Lemma~\ref{lem:unique-crossing} applies.
\end{proof}

\subsection{Return to the \texorpdfstring{$r$}{r}-axis}

Continue the orbit uniquely through $\pi/2$.  Let $s_1(\delta)$ be the first
positive arclength parameter at which $x=0$, $\theta=0$, or $\theta=\pi$; if
no such event occurs, take the maximal endpoint of the existence interval.
Before $s_1(\delta)$ we have
$x>0$, $0<\theta<\pi$, and $dr/ds>0$.

The following estimate controls both coordinates before the first stopping
event.

\begin{lemma}
\label{lem:small-global-bound}
For all sufficiently small $\delta$ and $0<s<s_1(\delta)$,
\begin{equation}\label{eq:small-global-bounds}
 0<x_\delta(s)\leq x_\delta(s_\delta^+)\longrightarrow0,
 \qquad 0<r_\delta(s)\leq C(n,k).
\end{equation}
\end{lemma}

\begin{proof}
After the nondegenerate crossing, Lemma~\ref{lem:unique-crossing} gives
$\theta>\pi/2$ locally.  This strict inequality persists until the stopping
event.  To see this without assuming monotonicity of $\theta$, put
\[
 \phi=\theta-\pi/2,\qquad u=\phi^k.
\]
Choose $\eta>0$ so small that
\[
 B\sin^{k-1}\eta<\frac12 r^{k+1}(s_\delta^+).
\]
On the post-crossing segment, $r\geq r(s_\delta^+)$ and $x>0$.  Therefore,
whenever $0<\phi\leq\eta$,
\[
 x\cos\phi+r\sin\phi-B\frac{\sin^k\phi}{r^k}
 \geq
 \frac{\sin\phi}{r^k}\bigl(r^{k+1}-B\sin^{k-1}\phi\bigr)>0.
\]
Equation \eqref{eq:H-def} gives $du/ds>0$ in this strip.  The orbit therefore
cannot cross the level $\phi=\eta$ downwards after leaving it, and cannot
return to $\phi=0$.  Hence $\pi/2<\theta<\pi$ before the stop.  It follows
that $dx/ds=\cos\theta<0$ after $s_\delta^+$, while $dx/ds>0$ before it; thus
\[
 0<x_\delta(s)\leq x_\delta(s_\delta^+)
\]
for $0<s<s_1(\delta)$.

For the radial bound, write $\zeta=-\cos\theta\in(0,1)$ on the post-crossing
segment and set $R_0=\max\{1,(2B)^{1/(k+1)}\}$.  Whenever $r\geq R_0$, the angular
equation gives
\begin{equation}\label{eq:small-post-cross-fast-angle}
 \frac{d\theta}{ds}
 =\frac{r^{k-1}}{A\zeta^{k-1}}
 \left(x\sin\theta+r\zeta-B\frac{\zeta^k}{r^k}\right)
 \geq\frac{r^k}{2A\zeta^{k-2}}\geq\frac{r^k}{2A}.
\end{equation}
Put
\[
 \overline R_\delta
 =\max\{R_0,r_\delta(s_\delta^+)\}.
\]
If the orbit never reaches $\overline R_\delta$, there is nothing to prove.
After its first contact with that radius, $r\geq\overline R_\delta$ and the
remaining change in $\theta$ before the stopping event is less than $\pi/2$.
Integration of \eqref{eq:small-post-cross-fast-angle}, first from a point
strictly after the singular crossing if necessary and then by passage to the
limit, shows that the remaining arclength is at most
$A\pi/\overline R_\delta^k$.  Since $0<dr/ds\leq1$,
\[
 r_\delta(s)
 \leq\overline R_\delta+\frac{A\pi}{\overline R_\delta^k}
 \leq\overline R_\delta+\frac{A\pi}{R_0^k}.
\]
Lemma~\ref{lem:small-crosses} completes the proof.
\end{proof}

These bounds force sufficiently small orbits to return to the $r$-axis before
their tangent angle reaches $\pi$.

\begin{lemma}
\label{lem:small-returns}
For every sufficiently small $\delta>0$, $s_1(\delta)<\infty$ and
\begin{equation}\label{eq:small-return-event}
 x_\delta(s_1(\delta))=0,\qquad
 0<\theta_\delta(s_1(\delta))<\pi.
\end{equation}
\end{lemma}

\begin{proof}
\medskip\noindent\textbf{Step 1: exclude entry into a fixed neighbourhood of
$\theta=\pi$.}
Suppose, for contradiction, that there is a sequence $\delta_j\downarrow0$
for which the orbit does not return as asserted.  Suppress the sequence index,
set $\varepsilon_\delta:=x_\delta(s_\delta^+)$, and write
$0<x\leq\varepsilon_\delta\to0$ and $r\leq C$, using
Lemma~\ref{lem:small-global-bound}.  On the post-crossing segment set
$\alpha=\pi-\theta$.

Suppose the orbit reaches $\alpha=\pi/6$.  At its first earlier contact with
$\alpha=\pi/4$, say at parameter value $s_a$ and radius $r_a$, one has
$\frac{d\alpha}{ds}(s_a)\leq0$, hence
$\frac{d\theta}{ds}(s_a)\geq0$.  Substitution of
$\theta=3\pi/4$ in the angular equation gives
\begin{equation}\label{eq:ra-lower}
 B\leq2^{(k-1)/2}\bigl(r_a^{k+1}+r_a^k x(s_a)\bigr),
\end{equation}
so $r_a\geq c_0>0$ for small $\delta$.  Let $s_b$ be the last contact with
$\alpha=\pi/4$ before the first contact $s_c$ with $\alpha=\pi/6$.  On
$[s_b,s_c]$,
\[
 \frac{\pi}{6}\leq\alpha\leq\frac{\pi}{4},\qquad
 c_0\leq r\leq C,\qquad 0<x\leq1.
\]
This is a compact subset of the regular ODE region, so $|d\theta/ds|\leq M$
there with $M$ independent of $\delta$.  Since the net change of $\theta$
is $\pi/12$,
\[
 s_c-s_b\geq\frac{\pi}{12M}.
\]
Moreover, $dx/ds=-\cos\alpha\leq-1/\sqrt2$ on this interval, forcing $x$ to
decrease by at least $\pi/(12\sqrt2 M)$, contrary to
$x\leq\varepsilon_\delta\to0$.  Thus the orbit cannot reach
$\alpha=\pi/6$ before returning, and consequently
\begin{equation}\label{eq:r-speed-lower}
 \theta\leq5\pi/6,\qquad \frac{dr}{ds}=\sin\theta\geq1/2.
\end{equation}
\medskip\noindent\textbf{Step 2: force a finite return to the $r$-axis.}
If the orbit never returned, \eqref{eq:r-speed-lower} would contradict the
uniform bound on $r$.

\medskip\noindent\textbf{Step 3: exclude termination before a stopping event.}
It remains to exclude termination at a finite parameter before any stopping
event occurs.  For each fixed orbit, the positive $u$ barrier established in
Lemma~\ref{lem:small-global-bound} prevents a return to the singular section
once the orbit has left a sufficiently small crossing neighbourhood.
Together with the uniform bounds on $x$ and $r$ and
\eqref{eq:r-speed-lower}, the orbit then remains in a compact subset of the
regular ODE region until either $x=0$ or a recorded angular event occurs.
The continuation theorem excludes a finite maximal parameter endpoint without such an
event.  Finally, \eqref{eq:r-speed-lower} excludes $\theta=\pi$ first.
Hence \eqref{eq:small-return-event} holds.
\end{proof}

\section{The shooting set and global compactness}\label{sec:compactness}

If some orbit already reaches $x=0$ and $\theta=\pi$ simultaneously, the
desired half-profile curve has been found.  We therefore assume that no such orbit
exists and define the strict shooting set
\begin{equation}\label{eq:shooting-set}
 \sS=\left\{\delta\in(0,r_{n,k}):
 s_1(\delta)<\infty,\ x_\delta(s_1(\delta))=0,\
 \theta_\delta(s_1(\delta))<\pi\right\}.
\end{equation}
Lemma~\ref{lem:small-returns} shows that $\sS\neq\varnothing$.  Set
\begin{equation}\label{eq:delta-star}
 \delta_*=\sup\sS.
\end{equation}
We do not assume that $\sS$ is an interval or that $\delta_*\in\sS$.

\begin{figure}[htbp]
  \centering
  \includegraphics[width=0.96\textwidth]{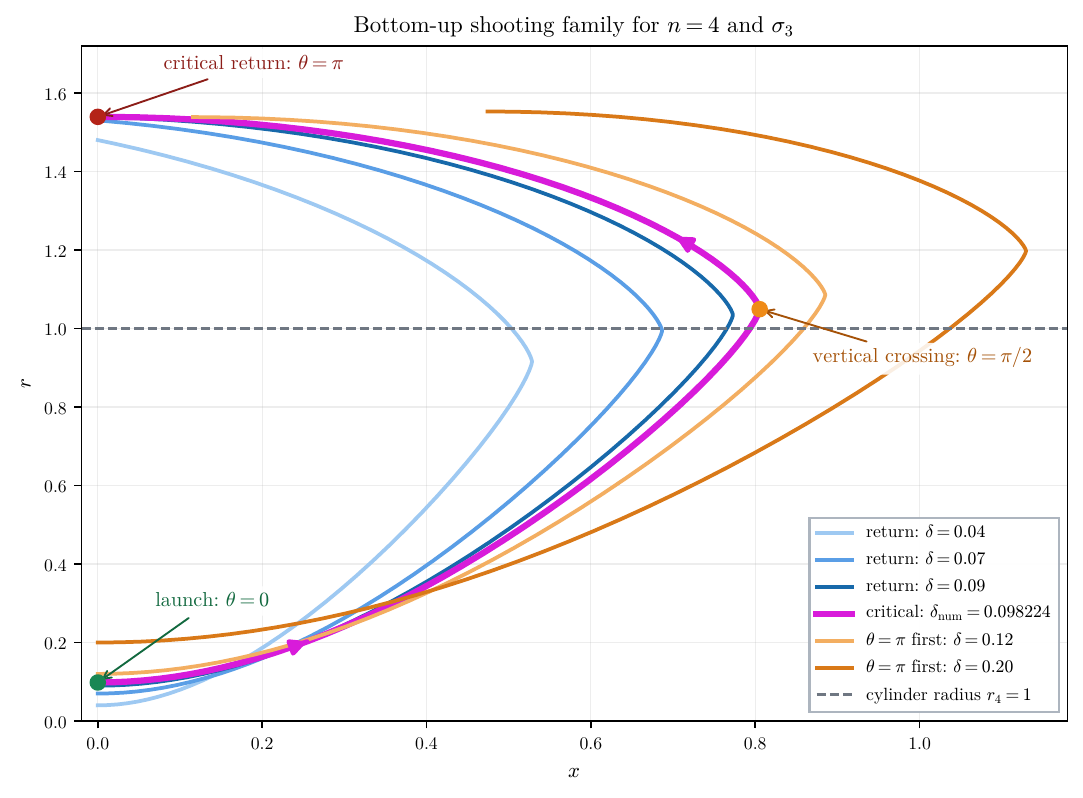}
  \caption{Numerical illustration of the bottom-up shooting family in the
  special case $k=3$ and $n=4$.  Each displayed profile curve is the
  $(x,r)$-projection of a solution orbit with initial state
  $(0,\delta,0)$, and the arrows indicate increasing arclength.  Blue profile
  curves return to $x=0$ before $\theta=\pi$, orange profile curves reach
  $\theta=\pi$ while $x>0$, and the magenta profile curve approximates the
  projection of the separating orbit at
  $\delta_{\mathrm{num}}\approx0.0982244941$.  The numerical value and the
  displayed profile curves are not used in the proof.}
  \label{fig:n4-shooting-family}
\end{figure}

We first record a consequence of the crossing analysis that will be used
throughout this section.  If $\delta\in\sS$, then the orbit must cross
$\theta=\pi/2$: otherwise $dx/ds=\cos\theta$ could not change sign and the
orbit could not return to $x=0$.  At its first crossing, $x>0$ and $r>0$, so
Lemma~\ref{lem:unique-crossing} applies.  This crossing is one-way.  Indeed,
with $\phi=\theta-\pi/2$ and $u=\phi^k$, choose $\eta>0$ so small that
\[
 B\sin^{k-1}\eta<\frac12 r^{k+1}(s_\delta^+).
\]
Since $r$ is increasing and $x>0$ before the terminal event, for
$0<\phi\leq\eta$ one has
\begin{equation}\label{eq:shooting-positive-u}
 x\cos\phi+r\sin\phi-B\frac{\sin^k\phi}{r^k}
 \geq
 \frac{\sin\phi}{r^k}\bigl(r^{k+1}-B\sin^{k-1}\phi\bigr)>0.
\end{equation}
Thus $du/ds>0$ in this strip.  The orbit cannot cross $\phi=\eta$ downwards
and cannot return to $\phi=0$.  Consequently every shooting orbit has a
unique vertical tangent, $x$ increases before it and decreases afterwards,
and its terminal angle belongs to $(\pi/2,\pi)$.

The one-way crossing and continuous dependence imply stability of the return
event under perturbation of the initial radius.

\begin{lemma}\label{lem:S-open}
The set $\sS$ is open in $(0,r_{n,k})$.
\end{lemma}

\begin{proof}
Fix $\bar\delta\in\sS$.  The corresponding orbit crosses $\pi/2$
nondegenerately.  Choose regular incoming and outgoing points with
$u=-u_0$ and $u=u_0$, and extend the reference orbit slightly beyond its
terminal parameter.  The regular flow from the initial point to the incoming
section, the fixed-section map of Lemma~\ref{lem:crossing-map}, and the
regular flow from the outgoing section to this final parameter value all depend
continuously on $\delta$.  At the reference terminal point,
\[
 x=0,\qquad \pi/2<\theta<\pi,\qquad \frac{dx}{ds}=\cos\theta<0.
\]
Choose parameters $s_-<s_1(\bar\delta)<s_+$ in the short regular extension such
that
\[
 x_{\bar\delta}(s_-)>0,\qquad x_{\bar\delta}(s_+)<0,
\]
while $\theta$ remains in a compact subinterval of $(\pi/2,\pi)$.
Continuous dependence preserves these strict inequalities for all
$\delta$ sufficiently close to $\bar\delta$.  The intermediate value theorem
then gives a zero of $x_\delta$ between the corresponding parameter values.

It remains to exclude an earlier stopping event.  On a common short initial
interval, $d\theta/ds(0)>0$ and $dx/ds(0)=1$ give $x>0$ and $\theta>0$ uniformly for
nearby initial radii.  Remove that interval and small fixed neighbourhoods of
the incoming and outgoing sections.  On each remaining compact piece before
$s_-$, the reference orbit has a positive distance from
\[
 x=0,\qquad \theta=0,\qquad \theta=\pi.
\]
This distance persists under regular continuous dependence; on the crossing
neighbourhood the fixed $u$-chart and
\eqref{eq:shooting-positive-u} give uniform separation from the stopping sets.
Finally, on $[s_-,s_+]$ the angle remains in a compact subinterval of
$(\pi/2,\pi)$ and $dx/ds<0$, so the zero supplied above is unique in that
window.  Thus it occurs before either angular stopping event, and every
sufficiently nearby parameter belongs to $\sS$.
\end{proof}

\subsection{The terminal-radius bound}

For shooting parameters bounded away from zero, the terminal radii are
uniformly bounded.

\begin{lemma}\label{lem:endpoint-radius}
For every $\delta_0\in(0,\delta_*)$ there is
$C_e=C_e(n,k,\delta_0)$ such that
\begin{equation}\label{eq:endpoint-radius-bound}
 r_\delta(s_1(\delta))\leq C_e
 \quad\text{for all }\delta\in\sS\cap[\delta_0,\delta_*).
\end{equation}
\end{lemma}

\begin{proof}
The initial points with $\delta\in[\delta_0,r_{n,k}]$ form a compact subset of
the regular region.  Uniform local existence gives $\eta>0$ such that
\[
 |\theta_\delta(s)|\leq\pi/3,\qquad
 \frac{dx_\delta}{ds}(s)=\cos\theta_\delta(s)\geq1/2
 \quad (0\leq s\leq\eta)
\]
for every initial radius in this interval.  Every shooting orbit reaches its
unique vertical tangent after this initial interval, and hence
\begin{equation}\label{eq:crossing-x-lower}
 x_\delta(s_\delta^+)\geq m:=\eta/2>0.
\end{equation}

Write $c=-\cos\theta>0$.  If
$r\geq R_0:=\max\{1,(2B)^{1/(k+1)}\}$, then
\begin{align}
 \frac{d\theta}{ds}
 &=\frac{r^{k-1}}{Ac^{k-1}}
 \left(x\sin\theta+rc-B\frac{c^k}{r^k}\right)\notag\\
 &\geq\frac{r^{k-1}}{Ac^{k-1}}\,c\left(r-\frac B{r^k}\right)
 \geq\frac{r^k}{2Ac^{k-2}}\geq\frac{r^k}{2A}.
 \label{eq:post-cross-fast-angle}
\end{align}

Put $R_\delta=r_\delta(s_\delta^+)$.

\medskip\noindent\textbf{Case 1: $R_\delta\geq R_0$.}
Then
$r_\delta(s)\geq R_\delta$ after the crossing.  Since the terminal angle is
less than $\pi$, integration of \eqref{eq:post-cross-fast-angle} gives
\[
 s_1(\delta)-s_\delta^+
 \leq\frac{A\pi}{R_\delta^k}.
\]
On the other hand, the orbit must lose the full amount
$x_\delta(s_\delta^+)$ before returning to $x=0$, while $|dx/ds|\leq1$.  Thus
\[
 m\leq x_\delta(s_\delta^+)
 \leq s_1(\delta)-s_\delta^+
 \leq\frac{A\pi}{R_\delta^k},
\]
and consequently $R_\delta\leq(A\pi/m)^{1/k}$.

\medskip\noindent\textbf{Case 2: $R_\delta<R_0$.}
The radius is bounded by $R_0$ until it first reaches
that value.  From that point onward,
\eqref{eq:post-cross-fast-angle} bounds the remaining arclength by
$A\pi/R_0^k$.  Combining the two cases and using $0<dr/ds\leq1$ gives the
explicit uniform estimate
\[
 r_\delta(s_1(\delta))
 \leq
 \max\left\{R_0,\left(\frac{A\pi}{m}\right)^{1/k}\right\}
 +\frac{A\pi}{R_0^k}.
\]
\end{proof}

\subsection{A strict gap from the cylinder}

Linearisation at the cylindrical solution excludes a full interval of
near-cylinder initial radii from the shooting set.

\begin{lemma}\label{lem:cylinder-gap}
There is $\varepsilon_0>0$ such that
\begin{equation}\label{eq:cylinder-gap}
 (r_{n,k}-\varepsilon_0,r_{n,k})\cap\sS=\varnothing.
\end{equation}
Consequently,
\begin{equation}\label{eq:delta-star-gap}
 0<\delta_*\leq r_{n,k}-\varepsilon_0<r_{n,k}.
\end{equation}
\end{lemma}

\begin{proof}
\medskip\noindent\textbf{Step 1: derive the exact scaled system.}
Write $\delta=r_{n,k}-\varepsilon$ and introduce
\begin{equation}\label{eq:cylinder-scaled-vars}
 \mathsf X_\varepsilon=x_\delta,\qquad
 \mathsf Y_\varepsilon=\frac{r_\delta-r_{n,k}}{\varepsilon},\qquad
 \widehat\Theta_\varepsilon=\frac{\theta_\delta}{\varepsilon}.
\end{equation}
The sans-serif letters distinguish the axial and radial perturbation variables
from the immersion $X$ and from the normalised coordinates introduced later;
the hat on $\widehat\Theta_\varepsilon$ distinguishes the scaled angle from
the unscaled tangent angle $\theta_\delta$.
In these variables the exact system is
\begin{equation}\label{eq:cylinder-full-scaled-system}
\begin{aligned}
 \frac{d\mathsf X_\varepsilon}{ds}
 &=\cos(\varepsilon\widehat\Theta_\varepsilon),\\
 \frac{d\mathsf Y_\varepsilon}{ds}
 &=\frac{\sin(\varepsilon\widehat\Theta_\varepsilon)}{\varepsilon},\\
 \frac{d\widehat\Theta_\varepsilon}{ds}
 &=\frac{(r_{n,k}+\varepsilon\mathsf Y_\varepsilon)^{k-1}}
 {A\cos^{k-1}(\varepsilon\widehat\Theta_\varepsilon)}
 \left[
 \mathsf X_\varepsilon
 \frac{\sin(\varepsilon\widehat\Theta_\varepsilon)}{\varepsilon}
 +G(\varepsilon,\mathsf Y_\varepsilon,\widehat\Theta_\varepsilon)
 \right],
\end{aligned}
\end{equation}
where
\begin{equation}\label{eq:cylinder-G}
 G(\varepsilon,Y,\Theta)
 :=
 \frac{1}{\varepsilon}
 \left[
 -(r_{n,k}+\varepsilon Y)\cos(\varepsilon\Theta)
 +\frac{r_{n,k}^{k+1}\cos^k(\varepsilon\Theta)}
 {(r_{n,k}+\varepsilon Y)^k}
 \right].
\end{equation}
Denote the bracketed numerator in \eqref{eq:cylinder-G} by
$F(\varepsilon,Y,\Theta)$, so that $G=F/\varepsilon$ for
$\varepsilon\neq0$.
The function $G$ is the angular numerator after subtracting the cylindrical
equilibrium and dividing by the perturbation size $\varepsilon$; introducing
it makes the first-order cylinder defect and its removable limit explicit.
The apparent quotients in \eqref{eq:cylinder-full-scaled-system} are
removable.  Indeed, $\sin(\varepsilon\Theta)/\varepsilon$ extends smoothly
with value $\Theta$ at $\varepsilon=0$.  The numerator defining $G$ is a
smooth function of $(\varepsilon,Y,\Theta)$ that vanishes identically at
$\varepsilon=0$; the integral identity
\[
 \frac{F(\varepsilon,Y,\Theta)}{\varepsilon}
 =\int_0^1
 \frac{\partial F}{\partial\varepsilon}
 (t\varepsilon,Y,\Theta)\,dt
\]
therefore gives a smooth extension, and direct differentiation yields
\begin{equation}\label{eq:cylinder-G-limit}
 G(0,Y,\Theta)=-(k+1)Y.
\end{equation}
Equivalently, since $B=r_{n,k}^{k+1}$,
\begin{equation}\label{eq:cylinder-expansion}
 -r\cos\theta+B\frac{\cos^k\theta}{r^k}
 =-(k+1)(r-r_{n,k})+O((r-r_{n,k})^2+\theta^2).
\end{equation}
\medskip\noindent\textbf{Step 2: solve the limiting cylinder perturbation.}
It follows from the displayed exact system that, on every fixed compact set
on which $r_{n,k}+\varepsilon Y>0$ and
$\cos(\varepsilon\Theta)\neq0$, the scaled vector field extends smoothly to
$\varepsilon=0$ and converges in $C^1$ to
\begin{equation}\label{eq:cylinder-limit-system}
 \frac{d\mathsf X}{ds}=1,\qquad
 \frac{d\mathsf Y}{ds}=\widehat\Theta,\qquad
 \frac{d\widehat\Theta}{ds}
 =c_{n,k}(\mathsf X\widehat\Theta-(k+1)\mathsf Y),\qquad
 c_{n,k}=\frac{r_{n,k}^{k-1}}{A},
\end{equation}
with $(\mathsf X,\mathsf Y,\widehat\Theta)(0)=(0,-1,0)$.  Since
$\mathsf X(s)=s$, the remaining equation is
\begin{equation}\label{eq:cylinder-Y}
 \frac{d^2\mathsf Y}{ds^2}
 =c_{n,k}\left(s\frac{d\mathsf Y}{ds}-(k+1)\mathsf Y\right),
 \qquad \mathsf Y(0)=-1,\qquad \frac{d\mathsf Y}{ds}(0)=0.
\end{equation}
Let $\operatorname{He}_m$ denote the probabilists' Hermite polynomial.  Since
$k+1$ is even, the solution is
\begin{equation}\label{eq:cylinder-polynomial}
 \mathsf Y(s)=-\frac{\operatorname{He}_{k+1}(\sqrt{c_{n,k}}\,s)}
 {\operatorname{He}_{k+1}(0)},\qquad
 \widehat\Theta=\frac{d\mathsf Y}{ds}.
\end{equation}
For completeness, the identities
$\frac{d}{dt}\operatorname{He}_m(t)=m\operatorname{He}_{m-1}(t)$ and
$\operatorname{He}_m(t)=t\operatorname{He}_{m-1}(t)
-(m-1)\operatorname{He}_{m-2}(t)$
show inductively that every $\operatorname{He}_m$ has $m$ distinct real
zeros: evaluation of the recurrence at consecutive zeros of
$\operatorname{He}_{m-1}$ gives alternating signs, and the two outer zeros
follow from the leading term; see also
Szeg\H{o}~\cite[Chapter~V]{Szego1975}.  Thus $\widehat\Theta$ has a first positive
zero $s_{\mathrm c}$.
Because $\frac{d\widehat\Theta}{ds}(0)=c_{n,k}(k+1)>0$ and this zero is simple,
\begin{equation}\label{eq:cylinder-transverse-zero}
 \widehat\Theta>0\text{ on }(0,s_{\mathrm c}),\qquad
 \widehat\Theta(s_{\mathrm c})=0,\qquad
 \frac{d\widehat\Theta}{ds}(s_{\mathrm c})<0,\qquad
 \mathsf X(s_{\mathrm c})=s_{\mathrm c}>0.
\end{equation}

\medskip\noindent\textbf{Step 3: transfer the first return of the angle.}
The $C^1$ convergence is important near the common zero at $s=0$.
To make the first-zero assertion explicit, choose $t_0>0$, $\eta>0$, and
$a_0,a_1>0$ so small that
\[
 \frac{d\widehat\Theta}{ds}\geq2a_0\quad\text{on }[0,t_0],
 \qquad
 \widehat\Theta>0\quad\text{on }[t_0,s_{\mathrm c}-\eta],
\]
and $d\widehat\Theta/ds<-2a_1$ on
$[s_{\mathrm c}-\eta,s_{\mathrm c}+\eta]$.  The $C^1$ convergence gives
\[
 \frac{d\widehat\Theta_\varepsilon}{ds}\geq a_0\quad\text{on }[0,t_0],
\]
so $\widehat\Theta_\varepsilon>0$ there after the common zero at $0$.  Uniform
$C^0$ convergence preserves positivity on $[t_0,s_{\mathrm c}-\eta]$, while
strict monotonicity on $[s_{\mathrm c}-\eta,s_{\mathrm c}+\eta]$ gives a
unique transverse zero $s_{\mathrm c,\varepsilon}\to s_{\mathrm c}$.  The
same convergence gives
\[
 \mathsf X_\varepsilon(s_{\mathrm c,\varepsilon})\longrightarrow s_{\mathrm c}>0,
\]
and, because $\widehat\Theta_\varepsilon$ stays uniformly bounded, also keeps
$0<\theta_\delta<\pi/2$ on $(0,s_{\mathrm c,\varepsilon})$.  Hence $dx_\delta/ds>0$
there and $r_\delta>0$.  Thus the first stopping event after the parameter
value zero is
$\theta_\delta(s_{\mathrm c,\varepsilon})=0$ at a point with $x_\delta>0$, rather than a
return to the $r$-axis.  Hence $\delta\notin\sS$ for every sufficiently small
$\varepsilon>0$.
\end{proof}

\subsection{An invariant slope barrier and axial compactness}

Before the first crossing, introduce
\begin{equation}\label{eq:normalised-vars}
 R=\frac r{r_{n,k}},\qquad \widehat X=\frac x{r_{n,k}},\qquad q=\tan\theta.
\end{equation}
Using $B=r_{n,k}^{k+1}$ and $B/A=\lambda$, system
\eqref{eq:profile-system} becomes
\begin{align}\label{eq:q-system}
 \frac{dR}{d\widehat X}&=q,\qquad
 \frac{dq}{d\widehat X}=\lambda R^{k-1}(1+q^2)^{(k+1)/2}h,
 \notag\\
 h&=\widehat Xq-R+\frac1{R^k(1+q^2)^{(k-1)/2}}.
\end{align}
Thus $q=dr/dx$ is the slope of the pre-crossing graph, while $h$ is the
normalised angular driving term: its sign is exactly the sign of
$dq/d\widehat X$.
The purpose of $h$ is therefore to turn monotonicity of the slope into a
first-contact barrier statement.

\begin{lemma}\label{lem:h-barrier}
In the region $q>0$, the set $\{h\leq0\}$ is forward invariant.  Every orbit
with parameter in $\sS$ satisfies $h>0$ before its first vertical tangent,
and therefore $dq/d\widehat X>0$ there.
\end{lemma}

\begin{proof}
Along \eqref{eq:q-system},
\begin{equation}\label{eq:h-derivative}
 \frac{dh}{d\widehat X}
 =\left(\widehat X-\frac{(k-1)q}{R^k(1+q^2)^{(k+1)/2}}\right)
 \frac{dq}{d\widehat X}
 -\frac{kq}{R^{k+1}(1+q^2)^{(k-1)/2}}.
\end{equation}
On $h=0$, one has $dq/d\widehat X=0$ and therefore
$dh/d\widehat X<0$.  The vector field points
strictly into $\{h<0\}$.  A standard first-contact argument therefore makes
$\{h\leq0\}$ forward invariant as long as $q>0$.  Inside that set,
$dq/d\widehat X\leq0$, so if an orbit enters it at
$\widehat X=\widehat X_0$, then
\begin{equation}\label{eq:h-region-q-bound}
 0<q(\widehat X)\leq q(\widehat X_0)
\end{equation}
for as long as $q>0$.  This bound also supplies the required continuation.
Indeed, since $dR/d\widehat X=q$,
\[
 R(\widehat X_0)\leq R(\widehat X)
 \leq R(\widehat X_0)+q(\widehat X_0)(\widehat X-\widehat X_0).
\]
Thus on every finite $\widehat X$-interval the pair $(R,q)$ remains in a compact
subset of $\{R>0,\ q\geq0\}$.  The right-hand side of
\eqref{eq:q-system} is smooth there, including at $q=0$, so the regular ODE
continuation theorem excludes a finite maximal $\widehat X$-value while $q$ remains
positive.  Consequently the forward orbit either reaches $q=0$ at a regular
point or exists for arbitrarily large $\widehat X$ with $q$ bounded by
\eqref{eq:h-region-q-bound}.  In neither case can it reach the vertical
tangent $q=+\infty$; in the former case it has reached $\theta=0$ first.

For an orbit launched with $\delta<r_{n,k}$, its initial value is
\[
 h(0)=-R(0)+R(0)^{-k}
 =\frac{1-R(0)^{k+1}}{R(0)^k}>0.
\]
If a shooting orbit had a first contact with $h=0$ before its vertical
tangent, the preceding invariance would prevent it from ever reaching
$q=+\infty$.  But the one-way crossing observation shows that every
$\delta\in\sS$ must reach a vertical tangent.  This contradiction proves
$h>0$, and the middle equation of \eqref{eq:q-system} then gives
$dq/d\widehat X>0$.
\end{proof}

Combining the one-sided barrier with the terminal-radius estimate gives the
required uniform bound in the axial direction.

\begin{proposition}\label{prop:x-compactness}
For every fixed pair $(n,k)$ satisfying \eqref{eq:nk-range} and every
$\delta_0\in(0,\delta_*)$, there is $C_x=C_x(n,k,\delta_0)$ such that
\begin{equation}\label{eq:x-compactness}
 \sup_{0<s<s_1(\delta)}x_\delta(s)\leq C_x
 \quad\text{for all }\delta\in\sS\cap[\delta_0,\delta_*).
\end{equation}
\end{proposition}

\begin{proof}
By \eqref{eq:delta-star-gap}, all relevant initial parameters lie in the
compact interval
\begin{equation}\label{eq:initial-compact-interval}
 I=[\delta_0,r_{n,k}-\varepsilon_0]\Subset(0,r_{n,k}).
\end{equation}
Equation \eqref{eq:initial-curvature} is strictly decreasing in $\delta$,
because
\[
 \frac{d}{d\delta}\frac{B-\delta^{k+1}}{A\delta}
 =-\frac{B+k\delta^{k+1}}{A\delta^2}<0.
\]
It therefore has a positive lower bound $m_0$ on $I$.  Uniform local
existence and continuity of the smooth vector field near the compact initial
set provide $\tau>0$, independent of $\delta\in I$, such that the solutions
stay in a common regular neighbourhood and
\begin{equation}\label{eq:uniform-initial-leaving}
 \frac{d\theta_\delta}{ds}(s)\geq m_0/2,\qquad
 0\leq\theta_\delta(s)\leq\pi/4
 \quad (0\leq s\leq\tau).
\end{equation}
Thus, for $q_0=\tan(m_0\tau/2)>0$,
\begin{equation}\label{eq:q0}
 q_\delta(\tau)\geq q_0,\qquad x_\delta(\tau)\leq\tau.
\end{equation}

Let $s_\delta^+$ be the first vertical tangent.  By
Lemma~\ref{lem:h-barrier}, $dq/d\widehat X>0$ along the portion with
$s\in[\tau,s_\delta^+)$.  Here $\widehat X$ is a valid independent variable
because $dx/ds=\cos\theta>0$ before the crossing.
Thus $q\geq q_0$.  Since $dr/ds=\sin\theta>0$ and the one-way crossing
observation gives
$r_\delta(s_\delta^+)\leq r_\delta(s_1(\delta))$, one has
\begin{align}
 x_\delta(s_\delta^+)-x_\delta(\tau)
 &=\int_{r_\delta(\tau)}^{r_\delta(s_\delta^+)}\frac{dr}{q_\delta(r)}
 \notag\\
 &\leq\frac{r_\delta(s_\delta^+)-r_\delta(\tau)}{q_0}
 \leq\frac{C_e}{q_0},\label{eq:x-integral-bound}
\end{align}
where Lemma~\ref{lem:endpoint-radius} was used in the last step.  The
one-way crossing observation also shows that $x$ increases before the
crossing and decreases afterwards.  Hence its global maximum is the crossing
value, bounded by
$\tau+C_e/q_0$.
\end{proof}

\begin{remark}
The constant in Proposition~\ref{prop:x-compactness} may depend on the fixed
pair $(n,k)$; no uniformity as either parameter varies is asserted.
\end{remark}

\section{The critical orbit}\label{sec:critical}

We now pass to the supremal shooting parameter.  The only nonstandard issue
is compactness through the singular angular section.

The compactness statement needed at the supremal parameter is the following.

\begin{lemma}\label{lem:critical-connection}
Let $\delta_j\in\sS$, $\delta_j<\delta_*$, and
$\delta_j\uparrow\delta_*$.  After passing to a subsequence, the corresponding
orbits converge to a limiting critical orbit segment.  In regular regions,
convergence is understood in the coordinates $(x,r,\theta)$, while near
$\theta=\pi/2$ we use the coordinate $u=(\theta-\pi/2)^k$.  Its terminal angle belongs
to $(\pi/2,\pi]$.  If that angle is strictly less than $\pi$, the transverse
return to the $r$-axis persists under perturbation of the initial radius.
\end{lemma}

\begin{proof}
\medskip\noindent\textbf{Step 1: uniform state and arclength bounds.}
Choose $\delta_0\in(0,\delta_*)$ and discard finitely many terms.  The bounds
of Lemma~\ref{lem:endpoint-radius} and Proposition~\ref{prop:x-compactness}
give
\begin{equation}\label{eq:xr-critical-bounds}
 0<r_j\leq C,\qquad 0\leq x_j\leq C.
\end{equation}
Since $x_j$ is monotone on each side of its unique vertical tangent and $r_j$
is increasing, let $\operatorname{Var}_{[0,s_{1,j}]}(f)$ denote the total
variation of $f$ on $[0,s_{1,j}]$.  Then
\begin{equation}\label{eq:arclength-critical-bound}
 s_{1,j}\leq\operatorname{Var}_{[0,s_{1,j}]}(x_j)
 +\operatorname{Var}_{[0,s_{1,j}]}(r_j)
 \leq2\sup x_j+r_j(s_{1,j})-\delta_j\leq C.
\end{equation}
The first inequality follows by integrating
\[
 1=\sqrt{\left(\frac{dx_j}{ds}\right)^2+
 \left(\frac{dr_j}{ds}\right)^2}
 \leq \left|\frac{dx_j}{ds}\right|+\left|\frac{dr_j}{ds}\right|.
\]
The one-way crossing observation in Section~\ref{sec:compactness} gives
$\operatorname{Var}_{[0,s_{1,j}]}(x_j)=2\max x_j$, while
$dr_j/ds>0$ gives
$\operatorname{Var}_{[0,s_{1,j}]}(r_j)=r_j(s_{1,j})-\delta_j$.

\medskip\noindent\textbf{Step 2: common crossing sections and post-crossing
separation.}
The first crossing is uniformly nondegenerate.  Indeed, uniform local
existence near the compact initial set gives $\tau>0$ such that
$|\theta_j|\leq\pi/3$ and $dx_j/ds\geq1/2$ on $[0,\tau]$.  Since the first
vertical tangent occurs later and $x_j$ is increasing before it,
\begin{equation}\label{eq:critical-crossing-lower}
 x_j(s_j^+)\geq x_j(\tau)\geq\tau/2,\qquad
 \delta_0\leq r_j(s_j^+)\leq C.
\end{equation}
Thus all crossing data belong to the compact set
\[
 K_0=[\tau/2,C]\times[\delta_0,C]
 \Subset\{(x,r):x>0,\ r>0\}.
\]
In the $u$-chart of Section~\ref{sec:crossing},
\[
 \frac{du}{ds}=H(x,r,u),\qquad H(x,r,0)=\frac{k r^{k-1}x}{A}.
\]
Thus $H$ has a common positive lower bound near all crossing data, and
Lemma~\ref{lem:crossing-map} supplies a common incoming section, a common
outgoing section, and continuous dependence of both the state and elapsed
arclength.

We also need a uniform post-crossing separation.  Choose
$\phi_0\in(0,\pi/4)$ such that
$B\sin^{k-1}\phi_0<\delta_0^{k+1}/2$.  For
$0<\phi=\theta-\pi/2\leq\phi_0$,
\begin{equation}\label{eq:positive-u-barrier}
 x\cos\phi+r\sin\phi-B\frac{\sin^k\phi}{r^k}
 \geq\frac{\sin\phi}{r^k}(r^{k+1}-B\sin^{k-1}\phi)>0.
\end{equation}
Hence $du/ds>0$.  On the common crossing chart one also has
$H\geq h_0>0$ and $|dx/du|\leq C_0$.  Choose
\[
 0<u_1<
 \min\left\{u_0,\phi_0^k,\frac{\tau}{4C_0}\right\}.
\]
The level $u=u_1$ is reached while $x_j\geq\tau/4$, because the loss in $x$
between $u=0$ and $u=u_1$ is at most $C_0u_1$.  No orbit can cross this
level downwards later: at a first downward crossing one would have
$du/ds\leq0$, contradicting \eqref{eq:positive-u-barrier}.  Therefore
\begin{equation}\label{eq:post-cross-separation}
 \theta_j\geq\pi/2+u_1^{1/k}
\end{equation}
after leaving the crossing chart.

\medskip\noindent\textbf{Step 3: compactness in the three fixed charts.}
We now spell out the three-chart compactness argument, including the varying
terminal parameters.  Let $a_j$ be the arclength parameter at which the orbit
reaches the incoming section $u=-u_0$, and let $b_j$ be the arclength parameter
at which it reaches the fixed
regular outgoing section $u=u_1$.  After passing to a subsequence, the bounds
above give
\[
 s_{1,j}\longrightarrow s_*,
 \qquad
 a_j\longrightarrow a_*,
\]
and the incoming data at $u=-u_0$ converge.  Before that section,
\[
 0\leq\theta_j\leq\pi/2-u_0^{1/k},\qquad
 \delta_0\leq r_j\leq C,\qquad 0\leq x_j\leq C.
\]
The regular vector field and its first derivatives are bounded on this
compact set.  Since the incoming section is a positive distance from the
singular angle, uniform local existence extends all pre-crossing segments a
common short arclength interval past $a_j$.  Standard continuous dependence
on the initial radius, applied on the resulting fixed parameter interval,
gives convergence of
the pre-crossing segments to a regular solution on $[0,a_*]$.  Evaluation at
$a_j\to a_*$ identifies its endpoint with the limiting incoming data.

On the fixed interval $[-u_0,u_1]$, Lemma~\ref{lem:crossing-map} gives
convergence of $(x_j,r_j,s_j)$ in the $u$-chart.  In particular,
\[
 b_j\longrightarrow b_*
\]
and the data at $u=u_1$ converge.  After reaching this section,
\eqref{eq:post-cross-separation} places the solution orbits in the compact
regular set
\[
 \begin{aligned}
 K_{\rm reg}
 &:=[0,C]\times[\delta_0,C]\times
 [\pi/2+u_1^{1/k},\pi],\\
 K_{\rm reg}&\Subset\{r>0,\ \cos\theta\neq0\}.
 \end{aligned}
\]
Set
\[
 \ell_j=s_{1,j}-b_j,\qquad \ell_*=s_*-b_*.
\]
Here $\ell_j$ is the post-crossing arclength of the $j$th orbit and
$\ell_*$ is its limit; these symbols are unrelated to the extinction time
$T$ in the homothetic flow.  Then $\ell_j\to\ell_*$.  Moreover,
$\ell_*>0$: at $u=u_1$ one has
$x_j\geq\tau/4$, while $|dx_j/ds|\leq1$ and $x_j(s_{1,j})=0$, so
$\ell_j\geq\tau/4$.

Write
\[
 Z_j(\sigma)=(x_j,r_j,\theta_j)(b_j+\sigma),
 \qquad 0\leq \sigma\leq \ell_j.
\]
Thus $Z_j$ is the post-crossing state curve translated to a common initial
parameter, and $\sigma$ is its translated arclength parameter.
Their terminal states belong to the compact regular set
\[
 K_{\rm end}
 =\{0\}\times[\delta_0,C]\times
 [\pi/2+u_1^{1/k},\pi].
\]
Because $K_{\rm end}$ is compact and separated from $r=0$ and
$\cos\theta=0$, uniform local existence provides $\eta>0$ and a compact set
\[
 K_{\rm ext}\Subset\{r>0,\ \cos\theta\neq0\}
\]
such that every regular solution starting in $K_{\rm end}$ is defined for
an arclength interval of length $\eta$ and remains in $K_{\rm ext}$.  Extend
each $Z_j$ from $\ell_j$ to $\ell_j+\eta$ and set
\[
 K_{\rm post}:=K_{\rm reg}\cup K_{\rm ext}.
\]
Thus each extended post-crossing orbit remains in the fixed compact regular
set $K_{\rm post}$.  The vector field is uniformly Lipschitz on a
neighbourhood of this set.  Since $\ell_j\to\ell_*$, for all large $j$ the
extended orbits are defined on the common interval
\[
 [0,L_0],\qquad L_0:=\ell_*+\eta/2,
\]
and continuous dependence on the convergent data at $u=u_1$ gives
\[
 Z_j\longrightarrow Z_*
 \quad\text{uniformly on }[0,L_0].
\]
For a state vector $Z=(x,r,\theta)$, let $[Z]_x$ denote its first, or axial,
component.  Since $\ell_j\to\ell_*$ and
$[Z_j(\ell_j)]_x=x_j(b_j+\ell_j)=x_j(s_{1,j})=0$, continuity and uniform
convergence give
\begin{equation}\label{eq:terminal-parameter-limit}
 [Z_*(\ell_*)]_x=0.
\end{equation}
The pre-crossing, $u$-chart, and post-crossing limits agree on their fixed
overlap sections and therefore form one limiting orbit segment
\[
 (x_*,r_*,\theta_*):[0,s_*]\longrightarrow
 [0,C]\times[\delta_0,C]\times[0,\pi].
\]
In the original arclength coordinate,
\eqref{eq:terminal-parameter-limit} is $x_*(s_*)=0$.
On the post-crossing segment, \eqref{eq:post-cross-separation} gives
$dx_*/ds=\cos\theta_*<0$.  Hence $x_*(s)>0$ for every post-crossing
$s<s_*$.

\medskip\noindent\textbf{Step 4: exclude an interior contact with
$\theta=\pi$.}
The limit cannot meet $\theta=\pi$ at an interior point with $x>0$.  At such a
point
\begin{equation}\label{eq:theta-at-pi}
 \frac{d\theta}{ds}=\frac{r^k}{A}\left(1-\frac{B}{r^{k+1}}\right).
\end{equation}
\medskip\noindent\textbf{Case 1: $r\neq r_{n,k}$.}
The intersection is transverse and would persist for the
approximating shooting orbits by regular continuous dependence, contradicting
$0<\theta_j<\pi$ before their terminal events.

\medskip\noindent\textbf{Case 2: $r=r_{n,k}$.}
The explicit
reverse cylinder
\[
 x=x_0-(s-s_0),\qquad r=r_{n,k},\qquad \theta=\pi
\]
passes through the point.  Let $[b_*,s_0]$ be the connected regular segment
from the fixed outgoing section to the alleged contact, and let
$Z_{\mathrm{cyl}}(s)$ denote
this reverse cylinder with the same state at $s_0$.  Define
\[
 E=\{s\in[b_*,s_0]:Z_*(s)=Z_{\mathrm{cyl}}(s)\}.
\]
The set $E$ is nonempty because $s_0\in E$, and it is closed by continuity.
At every $s\in E$ the vector field is smooth; local uniqueness both forward
and backward in the arclength parameter shows that $E$ contains a relative neighbourhood of
$s$.  Thus $E$ is also open in the connected interval $[b_*,s_0]$, so
$E=[b_*,s_0]$.  In particular, $\theta_*\equiv\pi$ at the outgoing section,
contradicting $\theta_*(b_*)=\pi/2+u_1^{1/k}<\pi$.

It follows from \eqref{eq:post-cross-separation} that the terminal angle lies
in $(\pi/2,\pi]$.  If it is less than $\pi$, then
$dx_*/ds(s_*)<0$ and the terminal
intersection with $x=0$ is transverse.  On the post-crossing segment one has
$dx_*/ds<0$.  The preceding exclusion of an interior contact with $\theta=\pi$,
together with the strict terminal inequality, also places the entire
post-crossing angle in a compact subinterval of $(\pi/2,\pi)$.

\medskip\noindent\textbf{Step 5: persistence of a transverse terminal return.}
Extend the limiting regular solution a short arclength interval past $s_*$ and choose
$s_-<s_*<s_+$ such that
\[
 x_*(s_-)>0,\qquad x_*(s_+)<0,
\]
while $\theta_*$ remains in a compact subinterval of $(\pi/2,\pi)$.
Choose a small fixed initial interval on which the uniform initial-departure
estimate gives $x>0$ and $\theta>0$.  From the end of that interval to the
incoming section, Lemma~\ref{lem:h-barrier} and the uniform lower bound for
$q$ give the same strict inequalities in the limit.  Across the crossing
chart $x$ stays positive, and on the post-crossing segment positivity was
proved above.  Hence, on the compact remainder of the limiting orbit up to
$s_-$, there is a positive distance from the stopping sets $x=0$,
$\theta=0$, and $\theta=\pi$.  The pre-crossing regular flow, the fixed $u$
crossing map, and the post-crossing regular flow compose to a map continuous
in the initial radius on these fixed sections and parameter intervals.  Therefore
the initial departure, the uniform separation from the stopping sets up to
$s_-$, the two
signs at $s_-$ and $s_+$, and the angular separation on $[s_-,s_+]$ all
persist for every sufficiently close initial radius.  The intermediate value
theorem, together with $dx/ds<0$ on the terminal window, then gives a unique
transverse return to $x=0$ before either angular stopping event.
\end{proof}

\begin{proof}[Proof of Theorem~\ref{thm:main}, geometric part]
By Lemma~\ref{lem:S-open}, the number $\delta_*<r_{n,k}$ cannot itself belong to
$\sS$; otherwise parameters slightly larger than $\delta_*$ would also
belong to $\sS$.  Choose
\[
 \delta_j\in\sS,\qquad \delta_j<\delta_*,\qquad
 \delta_j\uparrow\delta_*.
\]
Lemma~\ref{lem:critical-connection} gives a critical orbit with
\begin{equation}\label{eq:critical-end-data}
 (x_*(0),r_*(0),\theta_*(0))=(0,\delta_*,0),
 \qquad
 x_*(s_*)=0,\quad \theta_*(s_*)\in(\pi/2,\pi].
\end{equation}
If $\theta_*(s_*)<\pi$, the persistence statement of
Lemma~\ref{lem:critical-connection} places some $\delta>\delta_*$ in
$\sS$, contradicting \eqref{eq:delta-star}.  Hence
\begin{equation}\label{eq:critical-angle-pi}
 \theta_*(s_*)=\pi.
\end{equation}

On $(0,s_*)$ one has $x_*>0$ and $dr_*/ds>0$.  These strict inequalities follow
from the event structure, not merely from pointwise convergence.  Before the
vertical tangent, the uniform initial departure and
Lemma~\ref{lem:h-barrier} give $0<\theta_*<\pi/2$, hence
$dx_*/ds>0$ and $dr_*/ds>0$.  After the outgoing crossing section,
\eqref{eq:post-cross-separation} and the exclusion of an interior contact
with $\theta=\pi$ give $\pi/2<\theta_*<\pi$, so $dx_*/ds<0$ and
$dr_*/ds>0$.  A hypothetical interior zero of $x_*$ on this latter segment
would be transverse and would persist for the approximating shooting orbits
before their terminal parameters, a contradiction.  Thus the half-profile
curve is the positive graph $x=f(r)>0$.

Reflect it across the $r$-axis using Proposition~\ref{prop:reflection}.  The
endpoint tangents are horizontal and oppositely directed, so the reflected
curve is $C^1$ at its two intersections with the $r$-axis.  Since the
half-profile curve is a graph with strictly increasing $r$, it has no
self-intersections; its
reflected negative graph meets it only at the two endpoints.  The resulting
closed curve is therefore simple.

The two intersections with the $r$-axis are in fact smooth joining points.
On the reflected half use the real angle lift
\[
 \widetilde\theta(\bar s)=2\pi-\theta_*(s_*-\bar s).
\]
At the outer intersection with the $r$-axis this lift equals $\pi$, matching
the terminal state of the original half-profile curve.  At the inner
intersection with the $r$-axis it equals $2\pi$,
which matches the initial angle after replacing its real lift $0$ by
$2\pi$.  The profile system is smooth at both horizontal tangencies and is
$2\pi$-periodic in the angle.  Local uniqueness therefore identifies the two
sides at each intersection as restrictions of one smooth regular ODE orbit.

It follows that the closed curve is smooth except at the two reflected
vertical tangencies.  At these
points Lemma~\ref{lem:unique-crossing} gives the unique continuation and the
$C^{1,1/k}\cap W^{2,p}$ regularity for $p<k/(k-1)$.  More explicitly,
\eqref{eq:tangent-holder} applies near each singular parameter, while the
unit tangent field is smooth, and hence locally Lipschitz, at every regular
parameter.  These neighbourhoods form an open cover of the compact parameter
circle.  Let $d_{\mathrm{arc}}$ denote the shorter arclength distance between
two parameter values on this circle.  Choose a finite subcover and let
$\ell_{\mathrm L}>0$ be a Lebesgue number for it.  After enlarging a common
constant $C$, every pair with $d_{\mathrm{arc}}(s,t)<\ell_{\mathrm L}$
satisfies the corresponding local $1/k$-H\"older estimate.  If
$d_{\mathrm{arc}}(s,t)\geq\ell_{\mathrm L}$, the unit length of the two
tangent vectors gives the same estimate with constant
$2\ell_{\mathrm L}^{-1/k}$.  Thus, for the closed profile curve $\gamma$,
\[
 \left|
 \frac{d\gamma}{ds}(s)-\frac{d\gamma}{ds}(t)
 \right|
 \leq C\,d_{\mathrm{arc}}(s,t)^{1/k}.
\]
on the closed profile curve.  This proves the global $C^{1,1/k}$ assertion in
Theorem~\ref{thm:main}\textup{(iv)}.

Rotation about the $x$-axis
produces an embedded hypersurface: equality of two rotated points first gives
equality of their $(x,r)$ profile coordinates, hence equality of their profile
parameters, and then equality of their spherical variables because $r>0$.
Its topology is that of the product of unit spheres
$\Sn^1\times\Sn^{n-1}$.  It satisfies \eqref{eq:self-similar} classically
away from the two singular latitudes.  The almost-everywhere statement and
the flow interpretation are proved in Section~\ref{sec:weak}.

With the original half-profile curve followed by its reflected half, the closed
profile curve is counterclockwise in the $(x,r)$-plane.  Its outward planar normal
is therefore $(dr/ds,-dx/ds)=(\sin\theta,-\cos\theta)$, so the normal
$\nu=(\sin\theta,-\cos\theta\,\omega)$ used throughout is precisely the
outward normal of the bounded solid torus.
\end{proof}

\section{Sobolev almost-everywhere interpretation of the flow}\label{sec:weak}

Antonini~\cite[Definition~3 and equation~(2.8)]{Antonini2024} defines the weak
second fundamental form of a Lipschitz boundary whose local graph functions
lie in $W^{2,1}$.
For a graph $x_{n+1}=\varphi(y)$ with base variable $y\in\R^n$, his convention gives
\[
 \mathcal B_{ij}=-\frac{\partial^2\varphi/\partial y_i\partial y_j}
 {\sqrt{1+|\nabla\varphi|^2}}
\]
almost everywhere.  Here the domain lies locally below the graph and
$\nu=(-\nabla\varphi,1)/\sqrt{1+|\nabla\varphi|^2}$ is its outward unit
normal.  Thus $\mathcal B(U,V)=\langle D_U\nu,V\rangle$.  Since our shape
operator is $S=-D\nu$, its weak counterpart is
$S_{\mathrm w}=-g^{-1}\mathcal B$, where $g$ is the induced metric.

We use the following notion of solution for the constructed hypersurface.

\begin{definition}\label{def:Sobolev-ae}
Let $\Sigma=\partial\Omega$ be a compact embedded $C^1\cap W^{2,1}$
hypersurface with continuous unit normal $\nu$ and weak shape operator
$S_{\mathrm w}$.  It is a Sobolev almost-everywhere solution of
\eqref{eq:self-similar} if $\sigma_k(S_{\mathrm w})\in L^1(\Sigma)$ and
\[
 \langle X,\nu\rangle=-\sigma_k(S_{\mathrm w})
\]
at $\mathcal H^n$-almost every point.
Here $C^1\cap W^{2,1}$ means that, after a rigid rotation if necessary,
$\Sigma$ is locally the graph of a function belonging to
$C^1\cap W^{2,1}$.
\end{definition}

The torus constructed above satisfies this definition.

\begin{proposition}
\label{prop:weak-interpretation}
The torus constructed in Section~\ref{sec:critical} is a Sobolev
almost-everywhere solution.  In fact,
$\sigma_k(S_{\mathrm w})\in L^\infty(\Sigma)$.
\end{proposition}

\begin{proof}
Near a singular parameter $s_0$ one has
$dr/ds(s_0)=\sin\theta(s_0)=\pm1$.  The inverse function theorem for the
$C^1$ function $r(s)$ therefore permits us to write the profile curve as
$x=f(r)$.  On the punctured neighbourhood,
\[
 \frac{df}{dr}=\frac{dx/ds}{dr/ds}=\cot\theta,
 \qquad
 \frac{d^2f}{dr^2}=\frac1{dr/ds}\frac{d}{ds}\cot\theta
 =-\frac{d\theta/ds}{\sin^3\theta}.
\]
Lemma~\ref{lem:unique-crossing} gives
$|d\theta/ds(s)|=O(|s-s_0|^{-(k-1)/k})$.  Since $dr/ds(s_0)=\pm1$, the quantities
$|r-r_0|$ and $|s-s_0|$ are comparable near $s_0$, and hence
\begin{equation}\label{eq:weak-f-second}
 \left|\frac{d^2f}{dr^2}(r)\right|\leq C|r-r_0|^{-(k-1)/k}.
\end{equation}

The corresponding hypersurface patch is the graph
$y\mapsto f(|y|)$ over an annulus in $\R^n$ containing $|y|=r_0>0$.
Writing $\rho=|y|$, one has
\[
 \frac{\partial^2}{\partial y_i\partial y_j}\bigl(f(\rho)\bigr)
 =
 \frac{d^2f}{d\rho^2}(\rho)\frac{y_i y_j}{\rho^2}
 +\frac{(df/d\rho)(\rho)}{\rho}
 \left(\delta_{ij}-\frac{y_i y_j}{\rho^2}\right).
\]
Here $\delta_{ij}$ is the Kronecker delta, $df/d\rho$ is continuous and
bounded, $\rho$ is bounded away from zero, and
\eqref{eq:weak-f-second} belongs to $L^p$ for $p<k/(k-1)$.
The nonzero leading coefficient in
\eqref{eq:cross-asymptotic}--\eqref{eq:curvature-blowup} gives the matching
lower bound $|d^2f/dr^2(r)|\geq c|r-r_0|^{-(k-1)/k}$ after shrinking the punctured
neighbourhood.  Thus the range $p<k/(k-1)$ is sharp.
Consequently, each singular patch belongs to $W^{2,p}$ for those exponents
and, in particular, to $W^{2,1}$.  The remaining patches are smooth, so a finite covering gives the
claimed global regularity.  Since the hypersurface is compact, connected,
embedded, and $C^1$, the Jordan--Brouwer separation theorem identifies it
with the boundary of the bounded complementary component.  It therefore has
a global weak shape operator as in Definition~\ref{def:Sobolev-ae}.

On every smooth patch, weak derivatives agree with classical derivatives,
and hence $S_{\mathrm w}=S$ almost everywhere there.  Each singular latitude
is a copy of $\Sn^{n-1}$ inside the $n$-dimensional hypersurface, so the two
singular latitudes have zero $\mathcal H^n$ measure.  Therefore
\[
 \sigma_k(S_{\mathrm w})
 =\sigma_k(S)
 =-\langle X,\nu\rangle
\]
almost everywhere.  Although the individual coefficients of
$S_{\mathrm w}$ need not be bounded, this particular degree-$k$ combination
equals almost everywhere the continuous function $-\langle X,\nu\rangle$
on the compact hypersurface.  It consequently belongs to
$L^\infty(\Sigma)$ and in particular to $L^1(\Sigma)$.
\end{proof}

The self-similar equation now gives the corresponding almost-everywhere flow.

\begin{proposition}\label{prop:weak-flow}
For
\[
 a(t)=((k+1)(T-t))^{1/(k+1)},\qquad X_t=a(t)X,
\]
let $\nu_t$ and $S_{\mathrm w,t}$ denote the outward unit normal and weak
shape operator of $X_t(M)$, respectively.  Then
\[
 \left\langle\frac{\partial X_t}{\partial t},\nu_t\right\rangle
 =\sigma_k(S_{\mathrm w,t})
\]
at almost every spatial point for each $t<T$.
\end{proposition}

\begin{proof}
The scalar $a(t)$ is the homothetic scale factor and is unrelated to the
dimension constant $\lambda=B/A$ in \eqref{eq:lambda}.  The weak shape operators
satisfy $S_{\mathrm w,t}=a(t)^{-1}S_{\mathrm w}$ almost everywhere, so
$\sigma_k(S_{\mathrm w,t})=a(t)^{-k}\sigma_k(S_{\mathrm w})$.  Since
$da/dt=-a(t)^{-k}$,
\[
 \left\langle\frac{\partial X_t}{\partial t},\nu_t\right\rangle
 =\frac{da}{dt}\langle X,\nu\rangle
 =a(t)^{-k}\sigma_k(S_{\mathrm w})=\sigma_k(S_{\mathrm w,t})
\]
almost everywhere.
\end{proof}

\begin{remark}
This section makes no viscosity, level-set, Brakke, varifold, or general
curvature-measure-flow assertion.  The next section gives a more specific
statement about the rotational $\sigma_k$-density measure and a smooth
no-defect approximation;
neither statement asserts existence or uniqueness in any of those general
weak-flow frameworks.
\end{remark}

\section{Rotational \texorpdfstring{$\sigma_k$}{sigma-k}-density measure and
smooth approximation}
\label{sec:curvature-measure}

Let $d\mu_X$ denote the induced $n$-dimensional area measure.  The odd-power
crossing law allows the classical rotational density $\sigma_k\,d\mu_X$ to
extend continuously across the two singular latitudes.
We first construct this continuous-density extension and then give smooth
rotational approximations whose soliton residual tends to zero.  The
extension is not asserted to be unique among all Radon extensions or to agree
with every general notion of curvature measure in nonsmooth geometry.  In
particular, it differs in setting from the classical curvature measures of
Federer~\cite{Federer1959} and from the prescribed-curvature-measure problem
of Guan--Li--Li~\cite{GuanLiLi2012}, which is formulated for smooth
star-shaped admissible hypersurfaces.

At a singular parameter $s_0$, choose a local real lift of the tangent angle
and normalise
\[
 \theta_0:=\theta(s_0)\in\{\pi/2,-\pi/2\}.
\]
Put $\phi=\theta-\theta_0$ and $u=\phi^k$.  The angle is intrinsically
circle-valued.  If $L$ is the total arclength of the closed profile curve, a
real lift need only satisfy $\theta(s+L)=\theta(s)+2\pi m$ for some
$m\in\mathbb Z$; none of the expressions below requires a periodic
real-valued angle.

The odd-power crossing law produces a continuous curvature density and hence
no atomic contribution at either singular latitude.

\begin{proposition}
\label{prop:no-curvature-atom}
Let $Z\subset\Sn^1$ be the two singular parameter values of the closed profile
curve constructed above.  The quantity
\[
 \chi=\cos^{k-1}\theta\,\frac{d\theta}{ds}
\]
has a continuous extension through every $s_0\in Z$, with
\begin{equation}\label{eq:chi-crossing-value}
 \chi(s_0)=\frac{r(s_0)^{k-1}x(s_0)\sin\theta_0}{A}.
\end{equation}
Consequently the formula
\begin{equation}\label{eq:rotational-curvature-measure}
 d\mathcal K_k[X]
 =
 \left[
 -A r^{n-k}\chi+B r^{n-k-1}\cos^k\theta
 \right]\,ds\,d\omega
\end{equation}
defines a finite signed Radon measure on the fixed parameter manifold
$M=\Sn^1\times\Sn^{n-1}$, where $d\omega$ is the standard volume measure on
$\Sn^{n-1}$.  Moreover,
\begin{equation}\label{eq:renormalised-sigmak}
 \sigma_k^{\mathrm{ren}}
 =
 -A\frac{\chi}{r^{k-1}}
 +B\left(\frac{\cos\theta}{r}\right)^k
\end{equation}
is continuous on $M$ and satisfies
\[
 \sigma_k^{\mathrm{ren}}=-\langle X,\nu\rangle,
 \qquad
 \mathcal K_k[X]=\sigma_k^{\mathrm{ren}}\,d\mu_X.
\]
In particular,
\[
 \mathcal K_k[X](\{s_0\}\times\Sn^{n-1})=0
 \qquad (s_0\in Z).
\]
\end{proposition}

\begin{proof}
Since
$\cos(\theta_0+\phi)=-\sin\theta_0\sin\phi$ and
$\sin^2\theta_0=1$, away from $s_0$ one has
\begin{equation}\label{eq:odd-density-identity}
 \cos^{k-1}\theta\,\frac{d\theta}{ds}
 =
 \left(\frac{\sin\phi}{\phi}\right)^{k-1}\frac1k\frac{du}{ds}.
\end{equation}
The factor $\sin\phi/\phi$ extends continuously with value $1$.  The
odd-power
crossing construction gives $u\in C^1$ and
\[
 \frac{du}{ds}(s_0)=\frac{k r(s_0)^{k-1}x(s_0)\sin\theta_0}{A}.
\]
This proves the continuous extension and
\eqref{eq:chi-crossing-value}.

The functions $r$, $\cos\theta=dx/ds$, and the extended $\chi$ are continuous,
and the closed profile curve satisfies $r\geq r_{\min}>0$.  Hence the density in
\eqref{eq:rotational-curvature-measure} is continuous on the compact
parameter manifold and defines a finite signed Radon measure.  At every
smooth point, Proposition~\ref{prop:reduction} and
$d\mu_X=r^{n-1}ds\,d\omega$ give
\[
 \sigma_k\,d\mu_X
 =
 \left[
 -A r^{n-k}\cos^{k-1}\theta\,\frac{d\theta}{ds}
 +B r^{n-k-1}\cos^k\theta
 \right]\,ds\,d\omega,
\]
so this measure agrees with the classical curvature density there.

At $s_0\in Z$, equations \eqref{eq:chi-crossing-value} and
$\cos\theta_0=0$ give
\[
 \sigma_k^{\mathrm{ren}}(s_0)
 =-x(s_0)\sin\theta_0
 =-\langle X,\nu\rangle(s_0).
\]
The two sides are continuous and agree on the smooth locus, so they agree
everywhere.  Multiplication by $d\mu_X$ proves the measure identity.  Finally,
a continuous density with respect to $ds\,d\omega$ assigns zero mass to each
singular latitude.

On the other hand, Proposition~\ref{prop:weak-interpretation} gives
$S_{\mathrm w}=S$ almost everywhere off the singular latitudes, which are
$\mathcal H^n$-null.  Consequently
\[
 \mathcal K_k[X]=\sigma_k(S_{\mathrm w})\,d\mu_X
\]
as signed measures.  The continuous-density construction above shows that
this almost-everywhere curvature measure acquires no additional singular
part on either latitude.
\end{proof}

Agreement with the classical density on the smooth locus does not by itself
determine the extension.

\begin{remark}
\label{rem:curvature-measure-nonuniqueness}
The qualification in Proposition~\ref{prop:no-curvature-atom} is essential.
If one asks only for agreement with $\sigma_k\,d\mu_X$ on the smooth locus,
then for arbitrary constants $c_{s_0}$ one may add
\[
 \sum_{s_0\in Z}c_{s_0}\,
 \delta_{s_0}\otimes d\omega.
\]
Here $\delta_{s_0}$ denotes the Dirac mass at the profile parameter $s_0$.
Thus the continuous-density, no-singular-part prescription selects the
extension above, but the extension is not unique among all signed Radon
extensions.
\end{remark}

The constructed torus also admits smooth rotational approximations whose
soliton residual has no limiting defect measure.

\begin{proposition}
\label{prop:smooth-no-defect}
There are smooth embedded rotational tori
$X_j:M\to\R^{n+1}$ such that, for every $1\leq p<k/(k-1)$,
\[
 X_j\longrightarrow X
 \quad\text{in }C^1(M)\cap W^{2,p}(M).
\]
Let $S_{X_j}$ and $\nu_j$ denote the shape operator and outward unit normal
of $X_j$, respectively, and let $d\mu_{X_j}$ be its induced volume measure.
When the curvature and support-function measures are regarded as measures on
the fixed parameter manifold $M$,
\begin{align}
 \sigma_k(S_{X_j})\,d\mu_{X_j}
 &\longrightarrow\mathcal K_k[X],
 \label{eq:fixed-domain-curvature-TV}\\
 \langle X_j,\nu_j\rangle\,d\mu_{X_j}
 &\longrightarrow\langle X,\nu\rangle\,d\mu_X
 \label{eq:fixed-domain-support-TV}
\end{align}
in total variation.  In particular, the fixed-domain residual satisfies
\begin{equation}\label{eq:fixed-domain-residual-TV}
 \left\|
 \bigl(\sigma_k(S_{X_j})+\langle X_j,\nu_j\rangle\bigr)d\mu_{X_j}
 \right\|_{\mathrm{TV}}\longrightarrow0.
\end{equation}
Here $\|\cdot\|_{\mathrm{TV}}$ denotes the total variation norm.
After push-forward to $\R^{n+1}$, the two measures in
\eqref{eq:fixed-domain-curvature-TV}--\eqref{eq:fixed-domain-support-TV}
converge separately only weakly as Radon measures in general, whereas the
pushed-forward residual still converges to zero in total variation.
\end{proposition}

\begin{proof}
Let $\zeta$ denote an oriented periodic coordinate on $\Sn^1$, and
write $\gamma(\zeta)=(x(\zeta),r(\zeta))$.  Let
$\rho_{\varepsilon_j}$ be periodic approximate identities on $\Sn^1$, with
$\varepsilon_j\downarrow0$, and set
\[
 \gamma_j=\rho_{\varepsilon_j}*\gamma.
\]
Periodic convolution gives
\[
 \gamma_j\to\gamma
 \quad\text{in }C^1\cap W^{2,p}
 \quad\text{for every }1\leq p<k/(k-1).
\]
Because $\gamma$ is a regular embedded $C^1$ curve and
$\min r>0$, uniform $C^1$ convergence gives
$|d\gamma_j/d\zeta|\geq c>0$ and $r_j\geq r_{\min}/2$ for large $j$.
Embeddedness also persists.  For pairs of nearby parameters, uniform
continuity of $d\gamma/d\zeta$, together with projection onto the tangent
direction, gives a common linear lower bound for
$|\gamma_j(\zeta)-\gamma_j(\widetilde\zeta)|$.  For pairs separated by a
fixed positive circular distance, compactness and injectivity of $\gamma$
give a positive lower bound that persists under uniform convergence.  Thus $\gamma_j$ is
embedded for large $j$.  Positivity of $r_j$ then shows that
\[
 X_j(\zeta,\omega)=(x_j(\zeta),r_j(\zeta)\omega)
\]
is a smooth embedding.  In each chart of a fixed finite atlas on
$\Sn^{n-1}$, derivatives of $X_j$ of order at most two are linear
combinations of derivatives of $x_j,r_j$ of order at most two and fixed
smooth functions of $\omega$.  The profile convergence therefore also gives
$X_j\to X$ in $C^1(M)\cap W^{2,p}(M)$.

The parameter orientation is unchanged under this $C^1$ approximation:
the family of regular embedded curves obtained for sufficiently small
$\varepsilon_j$ is an isotopy through the same oriented parametrisation.
Hence the normals used below are the outward normals of the bounded solid
tori, consistently with the convention fixed above.

Put
\[
 v_j=\left|\frac{d\gamma_j}{d\zeta}\right|,\qquad
 p_j=\frac{dx_j/d\zeta}{v_j},\qquad
 \frac{d\theta_j}{d\zeta}=
 \frac{\det(d\gamma_j/d\zeta,d^2\gamma_j/d\zeta^2)}{v_j^2},
\]
and define $v$, $p$, and $d\theta/d\zeta$ analogously almost everywhere for
$\gamma$.  Thus $v_j$ is the speed of the profile parametrisation and $p_j$
is the axial component of its unit tangent.
Uniform convergence of the first derivatives, strong $L^1$ convergence of
the second derivatives, and the common positive lower bound for $v_j$ imply
\begin{equation}\label{eq:theta-dot-L1}
 \frac{d\theta_j}{d\zeta}\longrightarrow\frac{d\theta}{d\zeta}
 \quad\text{in }L^1(\Sn^1).
\end{equation}
Indeed, bilinearity gives
\[
 \begin{aligned}
 &\det(d\gamma_j/d\zeta,d^2\gamma_j/d\zeta^2)
 -\det(d\gamma/d\zeta,d^2\gamma/d\zeta^2)\\
 &\qquad
 =\det(d\gamma_j/d\zeta-d\gamma/d\zeta,d^2\gamma_j/d\zeta^2)
 +\det(d\gamma/d\zeta,
 d^2\gamma_j/d\zeta^2-d^2\gamma/d\zeta^2).
 \end{aligned}
\]
The first term tends to zero in $L^1$ because
$d\gamma_j/d\zeta-d\gamma/d\zeta\to0$ uniformly and
$\{d^2\gamma_j/d\zeta^2\}$ is bounded in $L^1$; the second tends to zero by
strong $L^1$ convergence of $d^2\gamma_j/d\zeta^2$.  Since $v_j^{-2}\to v^{-2}$
uniformly, \eqref{eq:theta-dot-L1} follows.

With respect to the fixed product measure $d\zeta\,d\omega$, the curvature
measure has density
\begin{equation}\label{eq:fixed-curvature-density}
 f_j=
 -A r_j^{n-k}p_j^{k-1}\frac{d\theta_j}{d\zeta}
 +B r_j^{n-k-1}p_j^k v_j,
\end{equation}
while the support-function measure has density
\begin{equation}\label{eq:fixed-support-density}
 g_j=\left(x_j\frac{dr_j}{d\zeta}
 -r_j\frac{dx_j}{d\zeta}\right)r_j^{n-1}.
\end{equation}
To verify these formulas, observe that
\[
 ds=v_j\,d\zeta,\qquad
 \cos\theta_j=p_j,\qquad
 \frac{d\theta_j}{ds}=\frac{1}{v_j}\frac{d\theta_j}{d\zeta}.
\]
Substitution in the smooth density from
Proposition~\ref{prop:reduction} gives
\eqref{eq:fixed-curvature-density}.  Moreover,
\[
 \nu_j=
 \left(\frac{dr_j/d\zeta}{v_j},
 -\frac{dx_j/d\zeta}{v_j}\omega\right),
 \qquad
 d\mu_{X_j}=r_j^{n-1}v_j\,d\zeta\,d\omega,
\]
which gives \eqref{eq:fixed-support-density}.
Equation \eqref{eq:theta-dot-L1} and uniform convergence of all coefficients
give $f_j\to f$ in $L^1$.  For its only second-derivative term, for example,
\[
 \left\|a_j\frac{d\theta_j}{d\zeta}
 -a\frac{d\theta}{d\zeta}\right\|_{L^1}
 \leq
 \|a_j-a\|_{L^\infty}\left\|\frac{d\theta_j}{d\zeta}\right\|_{L^1}
 +\|a\|_{L^\infty}
 \left\|\frac{d\theta_j}{d\zeta}-\frac{d\theta}{d\zeta}\right\|_{L^1},
\]
where $a_j=-Ar_j^{n-k}p_j^{k-1}$ and $a=-Ar^{n-k}p^{k-1}$; the right-hand
side tends to zero.
The same parameter change applied to
\eqref{eq:rotational-curvature-measure} shows that
$f\,d\zeta\,d\omega=\mathcal K_k[X]$.
The polynomial expression \eqref{eq:fixed-support-density} converges
uniformly to $g$, where
$g\,d\zeta\,d\omega=\langle X,\nu\rangle d\mu_X$.
This proves \eqref{eq:fixed-domain-curvature-TV} and
\eqref{eq:fixed-domain-support-TV}.  Proposition~\ref{prop:no-curvature-atom}
gives $f+g=0$, and hence the triangle inequality proves
\eqref{eq:fixed-domain-residual-TV}.

For the ambient assertion, define the fixed-domain signed measures
$K_j:=f_j\,d\zeta\,d\omega$ and
$H_j:=g_j\,d\zeta\,d\omega$, with limits
$K:=f\,d\zeta\,d\omega$ and $H:=g\,d\zeta\,d\omega$.  For a map $F$ and a
measure $\mu$, let $F_\#\mu$ denote the push-forward of $\mu$ by $F$.
Uniform convergence
$X_j\to X$ and total-variation
convergence on $M$ imply, for every compactly supported continuous test
function $\psi$,
\[
 \begin{aligned}
 \left|\int_M\psi(X_j)\,dK_j-\int_M\psi(X)\,dK\right|
 &\leq
 \|\psi\|_\infty\|K_j-K\|_{\mathrm{TV}}\\
 &\quad+
 \int_M|\psi(X_j)-\psi(X)|\,d|K|
 \longrightarrow0,
 \end{aligned}
\]
and similarly for $H_j$.  Thus the separate push-forwards converge weakly as
Radon measures.  On the other hand, push-forward contracts total variation, so
\[
 \bigl\|(X_j)_\#(K_j+H_j)\bigr\|_{\mathrm{TV}}
 \leq\|K_j+H_j\|_{\mathrm{TV}}\longrightarrow0.
\]
\end{proof}

Ambient total-variation convergence cannot in general be required for the
two separate surface-supported measures.

\begin{remark}
\label{rem:ambient-TV}
Even smooth convergence of moving embedded hypersurfaces does not imply
ambient total-variation convergence of their separate surface-supported
measures.  For example, translate a standard smooth rotational torus by
$\varepsilon_j$ in the axial direction.  The translated tori converge
smoothly in a fixed parametrisation, but their intersections with the
original torus have zero $n$-dimensional surface measure.  The corresponding
nonzero curvature measures are therefore mutually singular.  This shows why
Proposition~\ref{prop:smooth-no-defect} asserts ambient total-variation
convergence only for the residual, not for its two separate terms.
\end{remark}

\bibliographystyle{amsplain}
\bibliography{references}

\end{document}